\documentclass[onefignum,onetabnum,oneeqnum,onealgnum]{siamart251216}

\usepackage{amssymb}
\usepackage{algpseudocode}
\usepackage{graphicx}
\usepackage{booktabs}
\usepackage{placeins}
\usepackage{needspace}

\makeatletter
\DeclareRobustCommand{\ivpchanged}[1]{\leavevmode{\color{red}%
  \def\@linkcolor{red}\def\@citecolor{red}\def\@urlcolor{red}#1}}
\makeatother
\theoremstyle{plain}
\theoremheaderfont{\normalfont\itshape}
\theorembodyfont{\normalfont}
\theoremseparator{.}
\theoremsymbol{}
\newtheorem{assumption}{Assumption}
\crefname{assumption}{Assumption}{Assumptions}
\Crefname{assumption}{Assumption}{Assumptions}
\newsiamremark{example}{Example}

\newcounter{ivpresultpart}[theorem]

\newcommand{\ivpresultpart}{%
  \refstepcounter{ivpresultpart}%
  \textup{(\roman{ivpresultpart})}}
\newcounter{ivpassumptionpart}[assumption]

\newcommand{\ivpassumptionpart}{%
  \refstepcounter{ivpassumptionpart}%
  \textup{(\roman{ivpassumptionpart})}}

\newcommand{\TheTitle}{Projection-based parareal for nonlinear
  initial-value problems}
\newcommand{\TheAuthors}{C. Tan, Y. Xu, Y. Zhang, and Y. Su}
\headers{Projection-based parareal}{\TheAuthors}
\title{\TheTitle}
\author{
  Chenyi Tan\thanks{School of Mathematical Sciences, Fudan University,
    Shanghai, China (\email{cytan22@m.fudan.edu.cn},
    \email{xuyc23@m.fudan.edu.cn}, \email{yehaozhang23@m.fudan.edu.cn},
    \email{yfsu@fudan.edu.cn}).}
  \and
  Yuncheng Xu\footnotemark[1]
  \and
  Yehao Zhang\footnotemark[1]
  \and
  Yangfeng Su\footnotemark[1]
}

\ifpdf
\hypersetup{
  pdftitle={Projection-based parareal for nonlinear initial-value problems},
  pdfauthor={\TheAuthors}
}
\fi

\begin{document}

\maketitle

\begin{abstract}
\emph{Parareal} is a parallel-in-time method for initial-value problems
that has found extensive applications and given rise to many variants.
Krylov-enhanced \emph{parareal} is a variant of \emph{parareal} that uses
projection techniques to accelerate convergence substantially for linear
problems. In this paper, we study a
projection-based \emph{parareal} algorithm for nonlinear
initial-value problems with a modified iteration scheme and an arbitrary orthogonal projection.  For a general
nonlinear problem, we prove a convergence estimate for
each iteration step on a fixed time grid.
We analyze the errors of the modified scheme and the scheme of Krylov-enhanced
\emph{parareal} relative to the old error as the coarse grid is refined.
For a coarse solver of order \(p\) and step size \(h\), we prove error bounds
of \(O(h^{p+1})\) for the modified scheme and \(O(h)\) for that of
Krylov-enhanced \emph{parareal} at the first potentially inexact
time point under mild assumptions.
Numerical experiments on two nonlinear problems are consistent with the
theoretical analysis.
\end{abstract}



\section{Introduction}
\label{sec:introduction}

Consider the initial-value problem
\begin{equation}
  \label{eq:initial-value-problem}
  \mathbf u'(t)=\mathbf f(\mathbf u(t),t),
  \qquad t\in[0,T],
  \qquad \mathbf u(0)=\mathbf u_0,
\end{equation}
where \(\mathbf f:\mathbb R^d\times[0,T]\to\mathbb R^d\) is nonlinear and the
dimension \(d\) is fixed.  Problems of this form arise from nonlinear
dynamical models and from spatially discretized evolution
equations~\cite{hundsdorfer2003numerical,shu1988efficient,cockburn1989tvb}.
Representative examples include chemical reaction
kinetics~\cite{hairer1996solving}, pattern formation in reaction--diffusion
systems~\cite{turing1952chemical,prigogine1968symmetry}, molecular and
celestial dynamics~\cite{hairer2006geometric}, nerve
excitation~\cite{hodgkin1952quantitative}, and structural
dynamics~\cite{newmark1959method}.

For the numerical solution of initial-value problems, time integrators such
as Runge--Kutta~\cite{hairer1993solving,butcher2016numerical} and multistep
methods~\cite{hairer1993solving,hairer1996solving,ascher1998computer}
advance an approximation sequentially from one time point to the next.
The time discretization error depends on the step size and the order of
accuracy of the method.  The sequential structure of these methods can
lead to long computation times when small step sizes are required.
Many parallel-in-time integration methods instead compute on different time
intervals concurrently and enforce continuity across them by iteration.
Examples include the parallel full approximation scheme in space and time
(PFASST)~\cite{emmett2012toward}, multigrid reduction in time
(MGRIT)~\cite{falgout2014parallel}, and space--time multigrid
methods~\cite{horton1995spacetime,gander2016spacetime}.
Surveys of parallel-in-time integration are given
in~\cite{gander2015fifty,ong2020applications}.

The \emph{parareal} method is a parallel-in-time approach introduced
by Lions et al. for initial-value problems~\cite{lions2001parareal}.
The method combines parallel fine solves on separate subintervals
with a sequential coarse update of the approximation.
Each \emph{parareal} iteration updates the approximation over the entire time interval.
Convergence analyses for linear and nonlinear problems are given
in~\cite{bal2005convergence,gander2007parareal,gander2008nonlinear}.

Projection techniques have been widely used in variants of
\emph{parareal}, such as methods with reduced-order coarse
solvers~\cite{he2010reduced,chen2014parareal},
projection-stabilized methods~\cite{dai2013projection}, and low-rank
methods~\cite{carrel2023lowrank,gander2025lowrankcoarse,vargas2023multigrid}.
Krylov-enhanced \emph{parareal} was introduced by Gander and
Petcu~\cite{gander2008krylov} to improve the classical \emph{parareal}
iteration through projection techniques.
It uses a subspace constructed from solution histories
to enhance the coarse solver.
The idea goes back to the PITA framework of Farhat
et al.~\cite{farhat2006pita}.
Gander and Petcu~\cite{gander2008krylov} analyzed this method for linear
problems and reported substantially faster convergence than classical
\emph{parareal} in numerical experiments.

In this paper, we study a projection-based \emph{parareal} algorithm for nonlinear
initial-value problems with a modified iteration scheme and an arbitrary orthogonal projection.
Consider a time grid that divides \([0,T]\) into \(N\) coarse intervals
of equal length \(h=T/N\).
With this grid fixed, we prove a convergence estimate for each iteration
step of the algorithm.
The estimate separates the contribution generated by the old error on each
interval from its propagation to later time points and shows how both
depend on the orthogonal projection chosen for that iteration step.
In contrast, Krylov-enhanced \emph{parareal} introduces the projection
through a different iteration scheme.
We then derive asymptotic error estimates for the two schemes
relative to the old error as \(h\to0\).
For a coarse solver of order \(p\), we prove error estimates of
\(O(h^{p+1})\) for the modified scheme and \(\Theta(h)\) for the scheme of
Krylov-enhanced \emph{parareal} under mild assumptions.
Both refer to the first time point where the new approximation can be inexact.
At later time points, the corresponding upper bounds are
\(O(h^{p+1})\) and \(O(h)\) over a bounded number of coarse intervals, and
\(O(h^p)\) and \(O(1)\) over a fixed time window.

The paper is organized as follows. Section~\ref{sec:propagators} presents the algorithm after recalling
classical \emph{parareal} and Krylov-enhanced \emph{parareal}.
Section~\ref{sec:projection-dependent-analysis} analyzes convergence
as the iteration proceeds.
Section~\ref{sec:form-boundary} presents an asymptotic error analysis
as the coarse grid is refined.
Section~\ref{sec:numerical-experiments} reports numerical experiments.
Section~\ref{sec:conclusion} concludes the paper.

\section{Projection-based \emph{parareal}}
\label{sec:propagators}

This section first recalls classical \emph{parareal} and
Krylov-enhanced \emph{parareal}.
It then presents projection-based \emph{parareal} and compares its
iteration scheme with that of Krylov-enhanced \emph{parareal}.

\subsection{Classical \emph{parareal}}
\label{subsec:classical-parareal}

Consider the nonlinear initial-value
problem~\eqref{eq:initial-value-problem}.
Let \(N\ge2\) be an integer and define the uniform coarse time points by
\[
  T_n:=nh,\qquad h:=T/N,\qquad n=0,\ldots,N.
\]
On each coarse interval \([T_n,T_{n+1}]\), classical \emph{parareal}
uses a fine solver and a coarse solver.
For \(n=0,\ldots,N-1\), let \(\Omega_n\subseteq\mathbb R^d\) be a common
domain of initial values for the two solvers.
Assume that \(\mathbf0\in\Omega_n\) for every \(n\).
For \(\mathbf x\in\Omega_n\), let
\[
  \mathcal F(T_{n+1},T_n,\mathbf x),\qquad
  \mathcal G(T_{n+1},T_n,\mathbf x)
\]
denote the fine and coarse numerical approximations at \(T_{n+1}\) for
the local initial-value problem
\[
  \mathbf u'(t)=\mathbf f(\mathbf u(t),t),\qquad
  \mathbf u(T_n)=\mathbf x,\qquad t\in[T_n,T_{n+1}].
\]
The fine solver is chosen for accuracy, and the coarse solver
for lower computational cost.
For brevity, define \(\mathcal F_n,\mathcal G_n:\Omega_n\to\mathbb R^d\) by
\[
  \mathcal F_n(\mathbf x):=\mathcal F(T_{n+1},T_n,\mathbf x),\qquad
  \mathcal G_n(\mathbf x):=\mathcal G(T_{n+1},T_n,\mathbf x).
\]

Let \(\mathbf U^k=(\mathbf U_0^k,\ldots,\mathbf U_N^k)\in(\mathbb R^d)^{N+1}\)
denote an approximation after \(k\) iterations, with \(\mathbf U_n^k\) at the
coarse time point \(T_n\).
Starting from \(\mathbf U^0\) with \(\mathbf U_0^0=\mathbf u_0\),
the classical \emph{parareal} iteration~\cite{lions2001parareal,baffico2002parallel} computes
\begin{equation}
  \label{eq:classical-parareal}
  \begin{aligned}
    \mathbf U_0^{k+1}&=\mathbf u_0,\\
    \mathbf U_{n+1}^{k+1}
    &=\mathcal F_n(\mathbf U_n^k)
      +\mathcal G_n(\mathbf U_n^{k+1})
      -\mathcal G_n(\mathbf U_n^k),
  \end{aligned}
\end{equation}
for \(k=0,1,\ldots\) and \(n=0,\ldots,N-1\).
The solver values at the old approximation \(\mathbf U^k\) can be computed in parallel
across the coarse intervals.
The new approximation \(\mathbf U^{k+1}\) is computed sequentially.

Assume that the following recurrence is well defined:
\begin{equation}
  \label{eq:fine-reference}
  \mathbf u_0^\star=\mathbf u_0,\qquad
  \mathbf u_{n+1}^\star=\mathcal F_n(\mathbf u_n^\star),
  \qquad n=0,\ldots,N-1.
\end{equation}
The resulting sequence
\(\mathbf u^\star=(\mathbf u_0^\star,\ldots,\mathbf u_N^\star)\)
is called the \emph{fine reference solution}.
Iteration errors are measured relative to this reference:
\[
  \mathbf e_n^k:=\mathbf u_n^\star-\mathbf U_n^k,\qquad
  E_\infty^k:=\max_{0\le n\le N}\lVert\mathbf e_n^k\rVert_2,
\]
where \(\lVert\cdot\rVert_2\) denotes the Euclidean norm.

\subsection{Krylov-enhanced \emph{parareal}}
\label{subsec:krylov-enhanced-parareal}

To accelerate convergence of the classical \emph{parareal}
iteration, Krylov-enhanced
\emph{parareal}~\cite{gander2008krylov} introduces an
orthogonal projection into \eqref{eq:classical-parareal} by applying the fine
solver to the projected part of the argument and the coarse solver to
the complementary part.
The construction is first stated with \(\mathcal F_n\), \(\mathcal G_n\),
and an arbitrary orthogonal projection. The original method is then recovered
as a special case.
Let \(I\) denote the identity on \(\mathbb R^d\).
For an orthogonal projection \(P\), set \(Q:=I-P\) and define
\begin{equation}
  \label{eq:direct-split-propagator}
  \mathcal S_{P,n}(\mathbf x)
  :=\mathcal F_n(P\mathbf x)+\mathcal G_n(Q\mathbf x),
  \qquad \mathbf x,P\mathbf x,Q\mathbf x\in\Omega_n.
\end{equation}
At each iteration step \(k\to k+1\), choose an arbitrary orthogonal
projection \(P^k\) and keep it fixed across the coarse intervals.
The projection may change between iteration steps.
Replacing \(\mathcal G_n\) in \eqref{eq:classical-parareal} by
\(\mathcal S_{P^k,n}\) yields
\begin{equation}
  \label{eq:direct-split-iteration}
  \begin{aligned}
    \mathbf U_0^{k+1}&=\mathbf u_0,\\
    \mathbf U_{n+1}^{k+1}
    &=\mathcal F_n(\mathbf U_n^k)
      +\mathcal S_{P^k,n}(\mathbf U_n^{k+1})
      -\mathcal S_{P^k,n}(\mathbf U_n^k),
  \end{aligned}
\end{equation}
for \(n=0,\ldots,N-1\).
Set \(Q^k:=I-P^k\).
An iteration step \(k\to k+1\) of scheme \eqref{eq:direct-split-iteration} is called
\emph{admissible} if
\[
  \mathbf U_n^k,\ \mathbf U_n^{k+1},\ P^k\mathbf U_n^k,\ P^k\mathbf U_n^{k+1},\
  Q^k\mathbf U_n^k,\ Q^k\mathbf U_n^{k+1}\in\Omega_n,
  \qquad n=0,\ldots,N-1.
\]
All iteration steps of scheme \eqref{eq:direct-split-iteration} considered
in this paper are assumed admissible.
Krylov-enhanced \emph{parareal}~\cite{gander2008krylov} is formulated for
linear initial-value problems with linear fine and coarse solvers.
It takes \(P^k\) as the orthogonal projection onto the \emph{Krylov-enhanced subspace}
\[
  \mathcal V^k:=\operatorname{span}
  \{\mathbf U_j^\ell:0\le\ell\le k,\ 0\le j\le N\}.
\]
Then \(P^k\mathbf U_n^k=\mathbf U_n^k\) and \(Q^k\mathbf U_n^k=\mathbf0\).
Hence
\(\mathcal S_{P^k,n}(\mathbf U_n^k)=\mathcal F_n(\mathbf U_n^k)+\mathcal G_n(\mathbf0)\)
and \eqref{eq:direct-split-iteration} becomes
\[
  \mathbf U_{n+1}^{k+1}
  =\mathcal S_{P^k,n}(\mathbf U_n^{k+1})-\mathcal G_n(\mathbf0).
\]
With the sequential coarse initialization
\(\mathbf U_{n+1}^0=\mathcal G_n(\mathbf U_n^0)\), this is the scheme of
Krylov-enhanced \emph{parareal}~\cite[Sec.~4.3, Eq.~(4.10)]{gander2008krylov}.
For homogeneous linear fine and coarse solvers,
\(\mathcal G_n(\mathbf0)=\mathbf0\).
The scheme then reduces to that in~\cite[Sec.~4.1, Eqs.~(4.3) and~(4.5)]{gander2008krylov}.
Scheme \eqref{eq:direct-split-iteration} extends the scheme of Krylov-enhanced
\emph{parareal} to nonlinear solvers and an arbitrary orthogonal projection.

\subsection{A projection-based \emph{parareal} algorithm}
\label{subsec:projection-based-parareal}

We now introduce projection-based \emph{parareal}.
For an orthogonal projection \(P\) on \(\mathbb R^d\), define
\begin{equation}
  \label{eq:discrepancy-form-propagator}
  \mathcal H_{P,n}(\mathbf x)
  :=\mathcal G_n(\mathbf x)+\mathcal F_n(P\mathbf x)
                          -\mathcal G_n(P\mathbf x),
  \qquad \mathbf x,P\mathbf x\in\Omega_n.
\end{equation}
The map \(\mathcal H_{P,n}\) adds the discrepancy
\(\mathcal F_n(P\mathbf x)-\mathcal G_n(P\mathbf x)\) at \(P\mathbf x\)
to the coarse value \(\mathcal G_n(\mathbf x)\).
With \(P^k\) chosen at each iteration step as for scheme
\eqref{eq:direct-split-iteration}, replacing \(\mathcal G_n\) in
\eqref{eq:classical-parareal} by \(\mathcal H_{P^k,n}\) yields
\begin{equation}
  \label{eq:projection-parareal}
  \begin{aligned}
    \mathbf U_0^{k+1}&=\mathbf u_0,\\
    \mathbf U_{n+1}^{k+1}
    &=\mathcal F_n(\mathbf U_n^k)
      +\mathcal H_{P^k,n}(\mathbf U_n^{k+1})
      -\mathcal H_{P^k,n}(\mathbf U_n^k),
  \end{aligned}
\end{equation}
for \(n=0,\ldots,N-1\).
An iteration step \(k\to k+1\) of scheme \eqref{eq:projection-parareal} is called
\emph{admissible} if
\[
  \mathbf U_n^k,\ \mathbf U_n^{k+1},\ P^k\mathbf U_n^k,\ P^k\mathbf U_n^{k+1}
  \in\Omega_n,
  \qquad n=0,\ldots,N-1.
\]
All iteration steps of scheme \eqref{eq:projection-parareal} considered in
this paper are assumed admissible.

Once \(P^k\) is fixed, the values
\(\mathcal F_n(\mathbf U_n^k)\) and \(\mathcal H_{P^k,n}(\mathbf U_n^k)\)
can be computed in parallel across the coarse intervals.
The new approximation \(\mathbf U^{k+1}\) is computed sequentially from
\eqref{eq:projection-parareal}.
Algorithm~\ref{alg:projection-parareal} summarizes the method based on
scheme \eqref{eq:projection-parareal}.
\begin{algorithm}[H]
  \caption{Projection-based \emph{parareal}}
  \label{alg:projection-parareal}
  \begin{algorithmic}[1]
    \Require Fine and coarse solvers \(\mathcal F_n,\mathcal G_n\) for
      \(n=0,\ldots,N-1\) and an initial approximation \(\mathbf U^0\) with
      \(\mathbf U_0^0=\mathbf u_0\).
    \For{\(k=0,1,\ldots\)}
      \State Choose an orthogonal projection \(P^k\) and keep it fixed
        for the iteration step.
      \State Compute \(\mathcal F_n(\mathbf U_n^k)\) and
        \(\mathcal H_{P^k,n}(\mathbf U_n^k)\) for \(n=0,\ldots,N-1\)
        in parallel across the coarse intervals.
      \State Set \(\mathbf U_0^{k+1}=\mathbf u_0\) and compute
        \(\mathbf U_{n+1}^{k+1}\) for \(n=0,\ldots,N-1\) sequentially
        from \eqref{eq:projection-parareal}.
    \EndFor
  \end{algorithmic}
\end{algorithm}

Schemes \eqref{eq:direct-split-iteration} and \eqref{eq:projection-parareal}
define two variants of \emph{parareal} with an orthogonal projection.
Both involve the fine-solver value
\(\mathcal F_n(P^k\mathbf U_n^{k+1})\) in the sequential update.
For nonlinear solvers and \(P^k\ne0\), this value in general requires a
fine solve during the sequential update.
This paper analyzes the errors of the two schemes, not their computational
cost.
The construction of the projection space is not the subject of this paper.
Scheme \eqref{eq:projection-parareal} replaces
\(\mathcal S_{P^k,n}\) in \eqref{eq:direct-split-iteration}
by \(\mathcal H_{P^k,n}\).
For \(\mathbf x,P\mathbf x,Q\mathbf x\in\Omega_n\),
\begin{equation}
  \label{eq:exact-form-boundary}
  \mathcal H_{P,n}(\mathbf x)-\mathcal S_{P,n}(\mathbf x)
  =\mathcal G_n(\mathbf x)-\mathcal G_n(P\mathbf x)-\mathcal G_n(Q\mathbf x).
\end{equation}
If \(\mathcal G_n\) is linear, the right-hand side
is the constant \(-\mathcal G_n(\mathbf0)\) and cancels in both updates.
The two schemes then give the same new approximation.
For nonlinear \(\mathcal G_n\), the difference between
\(\mathcal G_n(P\mathbf x)+\mathcal G_n(Q\mathbf x)\) and
\(\mathcal G_n(\mathbf x)\) can depend on \(\mathbf x\).
The two schemes can then give different new approximations.
Section~\ref{sec:form-boundary} analyzes the errors of the two schemes
as \(h\to0\).

\section{Convergence analysis}
\label{sec:projection-dependent-analysis}

Throughout this section, \(N\) and \(h=T/N\) are fixed.
We establish finite-step exactness for schemes
\eqref{eq:direct-split-iteration} and \eqref{eq:projection-parareal}.
We then derive a convergence estimate for
Algorithm~\ref{alg:projection-parareal}.

\subsection{Finite-step exactness}
\label{subsec:finite-step-exactness}

Schemes \eqref{eq:direct-split-iteration} and \eqref{eq:projection-parareal}
reproduce the fine reference solution \(\mathbf u^\star\) on the time grid
after finitely many iteration steps: \(\mathbf U^k=\mathbf u^\star\)
and \(E_\infty^k=0\) for every \(k\ge N\).
We call this property finite-step exactness.
The following lemma makes it precise.
\begin{lemma}
  \leavevmode\phantomsection\label{lem:finite-step-exactness}\ignorespaces
  Let \(k\ge0\), and let \(\mathbf U^k\) satisfy
  \(\mathbf U_j^k=\mathbf u_j^\star\) for \(0\le j\le\min\{k,N\}\).
  Then one admissible step of scheme \eqref{eq:direct-split-iteration} or
  \eqref{eq:projection-parareal} gives
  \(\mathbf U_j^{k+1}=\mathbf u_j^\star\) for \(0\le j\le\min\{k+1,N\}\).
  Consequently, if \(\{\mathbf U^k\}_{k\ge0}\) is generated by admissible
  steps of either scheme from an initial approximation with
  \(\mathbf U_0^0=\mathbf u_0\), then
  \begin{equation}
    \label{eq:exactness-front}
    \mathbf e_j^k=\mathbf0,\qquad
    0\le j\le\min\{k,N\},\qquad k=0,1,\ldots,
  \end{equation}
  and \(E_\infty^k=0\) for every \(k\ge N\).
\end{lemma}
\begin{proof}
  We use induction on the time index \(j\).
  The new approximation starts from \(\mathbf U_0^{k+1}=\mathbf u_0\).
  Suppose that \(\mathbf U_j^{k+1}=\mathbf u_j^\star\) for some
  \(0\le j\le\min\{k,N-1\}\).
  The old approximation also satisfies \(\mathbf U_j^k=\mathbf u_j^\star\).
  Substituting \(\mathbf U_j^k=\mathbf U_j^{k+1}=\mathbf u_j^\star\)
  into \eqref{eq:projection-parareal}, we obtain
  \[
    \begin{aligned}
      \mathbf U_{j+1}^{k+1}
      &=\mathcal F_j(\mathbf U_j^k)
        +\mathcal H_{P^k,j}(\mathbf U_j^{k+1})
        -\mathcal H_{P^k,j}(\mathbf U_j^k)\\
      &=\mathcal F_j(\mathbf u_j^\star)
        +\mathcal H_{P^k,j}(\mathbf u_j^\star)
        -\mathcal H_{P^k,j}(\mathbf u_j^\star)\\
      &=\mathcal F_j(\mathbf u_j^\star)=\mathbf u_{j+1}^\star.
    \end{aligned}
  \]
  The same calculation applies to scheme \eqref{eq:direct-split-iteration},
  with \(\mathcal H_{P^k,j}\) replaced by \(\mathcal S_{P^k,j}\).
  This proves exactness through \(T_{\min\{k+1,N\}}\).
  Starting from \(\mathbf U_0^0=\mathbf u_0\), induction on \(k\) yields
  \eqref{eq:exactness-front}.
  The conclusion for \(k\ge N\) follows from
  \(E_\infty^k=\max_{0\le j\le N}\lVert\mathbf e_j^k\rVert_2\).
\end{proof}

\begin{samepage}
The proof uses the same cancellation argument as for classical
\emph{parareal}~\cite{baffico2002parallel}. For \(0\le k<N\), the cancellation in the proof extends exactness
from \(T_0,\ldots,T_k\) to \(T_0,\ldots,T_{k+1}\) in one admissible step.
Thus the fine reference solution \(\mathbf u^\star\) is obtained after at
most \(N\) iteration steps.
For \(0\le k\le N-2\), we call \(T_{k+2}\) the \emph{first potentially
inexact time point} after iteration step \(k\to k+1\) for both
schemes.
\par
\end{samepage}

\subsection{A convergence estimate for one iteration step}
\label{subsec:projection-dependent-one-sweep}

We analyze one iteration step of
Algorithm~\ref{alg:projection-parareal},
which is based on scheme \eqref{eq:projection-parareal}.
Fix an iteration step \(k\to k+1\) of the algorithm with \(0\le k\le N-2\),
and write \(P:=P^k\) and \(Q:=I-P\).
If \(E_\infty^k=0\), then \(E_\infty^{k+1}=0\) by the cancellation
argument in the proof of Lemma~\ref{lem:finite-step-exactness}.
We henceforth assume \(E_\infty^k>0\).
By Lemma~\ref{lem:finite-step-exactness}, \(\mathbf e_n^k=\mathbf0\) for
\(0\le n\le k\).
Thus the old error \(\mathbf e_n^k\) can enter the update only through
the intervals indexed by
\begin{equation}
  \label{eq:active-set}
  \mathcal A_k:=\{k+1,\ldots,N-1\}.
\end{equation}
For a differentiable map from an open subset of \(\mathbb R^d\)
into \(\mathbb R^d\), we write \(\mathrm D\) for its Jacobian.
For matrices, \(\lVert\cdot\rVert_2\) denotes the induced Euclidean norm.

Let \(\mathcal D_n:=\mathcal F_n-\mathcal G_n\) denote the discrepancy of
the two solvers.
Then \(\mathcal H_{P,n}=\mathcal G_n+\mathcal D_n\circ P\) by
\eqref{eq:discrepancy-form-propagator}.
For \(\mathbf x,P\mathbf x\in\Omega_n\), define the residual
\begin{equation}
  \label{eq:unresolved-residual}
  \mathcal R_{P,n}(\mathbf x)
  :=\mathcal F_n(\mathbf x)-\mathcal H_{P,n}(\mathbf x)
  =\mathcal D_n(\mathbf x)-\mathcal D_n(P\mathbf x).
\end{equation}
For \(n\in\mathcal A_k\) and \(0\le\tau\le1\), define
\[
  \mathbf z_n^k(\tau):=\mathbf U_n^k+\tau\mathbf e_n^k,\qquad
  \boldsymbol\zeta_n^{k+1}(\tau)
  :=\mathbf U_n^{k+1}+\tau\mathbf e_n^{k+1}.
\]
Then
\(\mathbf z_n^k(1)=\boldsymbol\zeta_n^{k+1}(1)=\mathbf u_n^\star\).
We impose the following assumption.
\begin{assumption}
  \label{ass:local-pathwise-regularity}
  For each \(n\in\mathcal A_k\), there is an open set
  \(\mathcal O_{n,k}\subseteq\Omega_n\) containing
  \[
    \{\mathbf z_n^k(\tau),\,P\mathbf z_n^k(\tau),\,
      \boldsymbol\zeta_n^{k+1}(\tau),\,
      P\boldsymbol\zeta_n^{k+1}(\tau):0\le\tau\le1\}.
  \]
  The fine and coarse solvers \(\mathcal F_n,\mathcal G_n\) belong to
  \(C^1(\mathcal O_{n,k};\mathbb R^d)\).
  There are finite constants \(\mu_k,\nu_k\ge0\) such that
  \begin{equation}
    \label{eq:source-path-bounds}
    \begin{aligned}
      \lVert\mathrm D\mathcal D_n(\mathbf z_n^k(\tau))\rVert_2
      &\le\mu_k,\\
      \lVert\mathrm D\mathcal D_n(\mathbf z_n^k(\tau))
        -\mathrm D\mathcal D_n(P\mathbf z_n^k(\tau))\rVert_2
      &\le\nu_k\lVert Q\mathbf z_n^k(\tau)\rVert_2,
    \end{aligned}
  \end{equation}
  for \(n\in\mathcal A_k\) and \(0\le\tau\le1\).
\end{assumption}

The constants \(\mu_k\) and \(\nu_k\) in \eqref{eq:source-path-bounds} are
local to this step and may depend on the grid, on \(P\), and on the old
approximation \(\mathbf U^k\). On the open set
\(\{\mathbf x\in\mathcal O_{n,k}:P\mathbf x\in\mathcal O_{n,k}\}\), the
chain rule gives
\[
  \mathrm D\mathcal H_{P,n}(\mathbf x)
  =\mathrm D\mathcal G_n(\mathbf x)
   +\mathrm D\mathcal D_n(P\mathbf x)P.
\]
This set contains \(\boldsymbol\zeta_n^{k+1}(\tau)\) for
\(0\le\tau\le1\).
Define the averaged Jacobian along \(\boldsymbol\zeta_n^{k+1}\) by
\begin{equation}
  \label{eq:realized-secant}
  B_n^{k+1}
  :=\int_0^1\mathrm D\mathcal H_{P,n}
     (\boldsymbol\zeta_n^{k+1}(\tau))\,d\tau,
  \qquad n\in\mathcal A_k.
\end{equation}
For \(k+1\le\ell\le m\le N\), set
\begin{equation}
  \label{eq:ordered-products}
  \Phi_{m,\ell}^{k+1}
  :=\begin{cases}
     B_{m-1}^{k+1}B_{m-2}^{k+1}\cdots B_\ell^{k+1},&m>\ell,\\
     I,&m=\ell.
    \end{cases}
\end{equation}
In \(\Phi_{m,\ell}^{k+1}\), the factors \(B_n^{k+1}\) act in temporal
order from right to left.
We define the propagation factor by
\begin{equation}
  \label{eq:causal-gain}
  \gamma_k
  :=\max_{k+2\le m\le N}
     \sum_{j=k+1}^{m-1}\lVert\Phi_{m,j+1}^{k+1}\rVert_2.
\end{equation}
Continuity of \(\mathrm D\mathcal H_{P,n}\) on the compact sets
\(\boldsymbol\zeta_n^{k+1}([0,1])\) and finiteness of \(\mathcal A_k\)
imply \(\gamma_k<\infty\).
The identity term \(\Phi_{m,m}^{k+1}=I\) gives \(\gamma_k\ge1\).
The factor \(\gamma_k\) in \eqref{eq:causal-gain} is an a posteriori
quantity, determined by \(P\), the new approximation
\(\mathbf U^{k+1}\), and the fine reference solution \(\mathbf u^\star\).
Assumption~\ref{ass:local-pathwise-regularity} provides no a priori bound
for \(\gamma_k\) that is uniform in the grid or the iteration step.

Define the ratios
\[
  \theta_Q^k(P)
  :=\frac{\max_{j\in\mathcal A_k}\lVert Q\mathbf e_j^k\rVert_2}
          {E_\infty^k},\qquad
  \theta_P^k(P)
  :=\frac{\max_{j\in\mathcal A_k}\lVert P\mathbf e_j^k\rVert_2}
          {E_\infty^k}.
\]
Set
\[
  U_k^\perp(P):=\max_{j\in\mathcal A_k}\lVert Q\mathbf U_j^k\rVert_2.
\]
By orthogonality, \(0\le\theta_Q^k(P)\le1\) and
\(0\le\theta_P^k(P)\le1\).
The maxima defining the numerators of \(\theta_Q^k(P)\) and
\(\theta_P^k(P)\) are taken over \(\mathcal A_k\), whereas the maximum
defining \(E_\infty^k\) also includes \(T_N\).
The following theorem bounds the new error \(E_\infty^{k+1}\)
in terms of the old error \(E_\infty^k\), with a factor that depends on
the projection \(P\).
\begin{theorem}
  \label{lem:projection-dependent-one-sweep}
  Let \(0\le k\le N-2\), and let \(\mathbf U^0,\ldots,\mathbf U^{k+1}\) be
  generated by admissible steps of Algorithm~\ref{alg:projection-parareal}.
  Set \(P:=P^k\) and \(Q:=I-P\).
  If \(E_\infty^k>0\) and Assumption~\ref{ass:local-pathwise-regularity}
  holds for the iteration step \(k\to k+1\), then
  \begin{equation}
    \label{eq:nonlinear-one-sweep-estimate}
    E_\infty^{k+1}\le\rho_k(P)E_\infty^k,
  \end{equation}
  where
  \begin{equation}
    \label{eq:nonlinear-one-sweep-factor}
    \rho_k(P):=\gamma_k\left[
         \mu_k\theta_Q^k(P)
         +\nu_kU_k^\perp(P)\theta_P^k(P)
         +\frac{\nu_k}{2}\theta_Q^k(P)\theta_P^k(P)E_\infty^k
       \right].
  \end{equation}
\end{theorem}
\begin{proof}
  By Assumption~\ref{ass:local-pathwise-regularity} at \(\tau=1\),
  \(\mathbf u_n^\star,P\mathbf u_n^\star\in\Omega_n\) for \(n\in\mathcal A_k\).
  Subtracting \eqref{eq:projection-parareal} from \eqref{eq:fine-reference}
  and using \(\mathcal F_n=\mathcal H_{P,n}+\mathcal R_{P,n}\), we obtain
  \begin{equation}
    \label{eq:corrected-error-identity}
    \begin{aligned}
      \mathbf e_{n+1}^{k+1}
      &=\mathcal H_{P,n}(\mathbf u_n^\star)
        -\mathcal H_{P,n}(\mathbf U_n^{k+1})+\mathcal R_{P,n}(\mathbf u_n^\star)
        -\mathcal R_{P,n}(\mathbf U_n^k),
      \qquad n\in\mathcal A_k.
    \end{aligned}
  \end{equation}
  By \eqref{eq:realized-secant} and the fundamental theorem of calculus
  along \(\boldsymbol\zeta_n^{k+1}\),
  \[
    \mathcal H_{P,n}(\mathbf u_n^\star)
     -\mathcal H_{P,n}(\mathbf U_n^{k+1})
    =B_n^{k+1}\mathbf e_n^{k+1}.
  \]
  For \(n\in\mathcal A_k\), define
  \begin{equation}
    \label{eq:old-state-source}
    \mathbf s_{n+1}^k
    :=\mathcal R_{P,n}(\mathbf u_n^\star)
      -\mathcal R_{P,n}(\mathbf U_n^k).
  \end{equation}
  By these identities and Lemma~\ref{lem:finite-step-exactness}, the
  new errors satisfy
  \begin{equation}
    \label{eq:causal-error-recurrence}
    \mathbf e_{n+1}^{k+1}
    =B_n^{k+1}\mathbf e_n^{k+1}+\mathbf s_{n+1}^k,
    \qquad n\in\mathcal A_k,\qquad
    \mathbf e_{k+1}^{k+1}=\mathbf0.
  \end{equation}
  Forward substitution gives
  \begin{equation}
    \label{eq:active-green-representation}
    \mathbf e_m^{k+1}
    =\sum_{j=k+1}^{m-1}
        \Phi_{m,j+1}^{k+1}\mathbf s_{j+1}^k,
    \qquad m=k+2,\ldots,N.
  \end{equation}
  For \(m=k+2,\ldots,N\), the indices \(j=k+1,\ldots,m-1\) in
  \eqref{eq:active-green-representation} lie in \(\mathcal A_k\), and taking
  norms yields
  \[
    \begin{aligned}
      \lVert\mathbf e_m^{k+1}\rVert_2
      &\le\sum_{j=k+1}^{m-1}\lVert\Phi_{m,j+1}^{k+1}\rVert_2
         \lVert\mathbf s_{j+1}^k\rVert_2\\
      &\le\Bigl(\sum_{j=k+1}^{m-1}\lVert\Phi_{m,j+1}^{k+1}\rVert_2\Bigr)
         \max_{j\in\mathcal A_k}\lVert\mathbf s_{j+1}^k\rVert_2\\
      &\le\gamma_k\max_{j\in\mathcal A_k}\lVert\mathbf s_{j+1}^k\rVert_2,
    \end{aligned}
  \]
  where the last inequality uses the definition \eqref{eq:causal-gain} of
  \(\gamma_k\).
  Since \(\mathbf e_m^{k+1}=\mathbf0\) for \(0\le m\le k+1\) by
  Lemma~\ref{lem:finite-step-exactness}, taking the maximum over
  \(0\le m\le N\) gives
  \begin{equation}
    \label{eq:causal-source-bound}
    E_\infty^{k+1}
    \le\gamma_k\max_{j\in\mathcal A_k}\lVert\mathbf s_{j+1}^k\rVert_2.
  \end{equation}

  To estimate \(\mathbf s_{n+1}^k\), use \eqref{eq:old-state-source},
  \eqref{eq:unresolved-residual}, and the fundamental theorem of calculus
  along \(\mathbf z_n^k\) and \(P\mathbf z_n^k\):
  \[
   \begin{aligned}
    \mathbf s_{n+1}^k
    &=\int_0^1\left[
        \mathrm D\mathcal D_n(\mathbf z_n^k(\tau))\mathbf e_n^k
        -\mathrm D\mathcal D_n(P\mathbf z_n^k(\tau))P\mathbf e_n^k
       \right]\,d\tau\\
    &=\int_0^1\mathrm D\mathcal D_n(\mathbf z_n^k(\tau))
                     Q\mathbf e_n^k\,d\tau+\int_0^1\left[
        \mathrm D\mathcal D_n(\mathbf z_n^k(\tau))
        -\mathrm D\mathcal D_n(P\mathbf z_n^k(\tau))
       \right]P\mathbf e_n^k\,d\tau.
   \end{aligned}
  \]
  By \eqref{eq:source-path-bounds}, the integral containing
  \(Q\mathbf e_n^k\) has norm at most \(\mu_k\lVert Q\mathbf e_n^k\rVert_2\).
  For the second integral, the second inequality in
  \eqref{eq:source-path-bounds}, the identity
  \(Q\mathbf z_n^k(\tau)=Q\mathbf U_n^k+\tau Q\mathbf e_n^k\), and the triangle
  inequality imply
  \[
    \begin{aligned}
    &\Bigl\lVert\int_0^1\bigl[\mathrm D\mathcal D_n(\mathbf z_n^k(\tau))
       -\mathrm D\mathcal D_n(P\mathbf z_n^k(\tau))\bigr]P\mathbf e_n^k\,d\tau
     \Bigr\rVert_2\\
    \qquad&\le\nu_k\lVert P\mathbf e_n^k\rVert_2
       \int_0^1\lVert Q\mathbf z_n^k(\tau)\rVert_2\,d\tau\\
    \qquad&\le\nu_k\lVert P\mathbf e_n^k\rVert_2
       \int_0^1\bigl(\lVert Q\mathbf U_n^k\rVert_2
          +\tau\lVert Q\mathbf e_n^k\rVert_2\bigr)\,d\tau\\
    \qquad&=\nu_k\Bigl(\lVert Q\mathbf U_n^k\rVert_2
         +\frac12\lVert Q\mathbf e_n^k\rVert_2\Bigr)
       \lVert P\mathbf e_n^k\rVert_2 .
    \end{aligned}
  \]
  By the triangle inequality and the bounds on the two integrals, we
  have
  \begin{equation}
    \label{eq:three-term-source-bound}
    \begin{aligned}
     \lVert\mathbf s_{n+1}^k\rVert_2
     &\le\mu_k\lVert Q\mathbf e_n^k\rVert_2+\nu_k\lVert Q\mathbf U_n^k\rVert_2
                      \lVert P\mathbf e_n^k\rVert_2+\frac{\nu_k}{2}\lVert Q\mathbf e_n^k\rVert_2
                      \lVert P\mathbf e_n^k\rVert_2.
    \end{aligned}
  \end{equation}
  Taking the maximum over \(n\in\mathcal A_k\) in
  \eqref{eq:three-term-source-bound} and using the definitions of
  \(\theta_Q^k(P)\), \(\theta_P^k(P)\), and \(U_k^\perp(P)\), we obtain
  \begin{equation}
    \label{eq:normalized-source-bound}
    \max_{j\in\mathcal A_k}\lVert\mathbf s_{j+1}^k\rVert_2
    \le\bigl[
      \mu_k\theta_Q^k(P)
      +\nu_kU_k^\perp(P)\theta_P^k(P)
      +\tfrac{\nu_k}{2}\theta_Q^k(P)\theta_P^k(P)E_\infty^k
    \bigr]E_\infty^k.
  \end{equation}
  Combining \eqref{eq:normalized-source-bound} with
  \eqref{eq:causal-source-bound} proves
  \eqref{eq:nonlinear-one-sweep-estimate}.
\end{proof}

The three terms in the bracket of \eqref{eq:nonlinear-one-sweep-factor}
show how the projection affects the iteration step \(k\to k+1\).
The first term \(\mu_k\theta_Q^k(P)\) measures the relative size of the
components \(Q\mathbf e_j^k\) outside the range of \(P\), weighted by
\(\mu_k\).
The second term \(\nu_kU_k^\perp(P)\theta_P^k(P)\) contains the largest
distance \(U_k^\perp(P)\) from the old approximations \(\mathbf U_j^k\) with
\(j\in\mathcal A_k\) to the range of \(P\).
It vanishes when \(P\mathbf U_j^k=\mathbf U_j^k\) for all \(j\in\mathcal A_k\).
The third term \(\tfrac{\nu_k}{2}\theta_Q^k(P)\theta_P^k(P)E_\infty^k\)
carries an additional factor \(E_\infty^k\), so its contribution to
\(E_\infty^{k+1}\) is quadratic in \(E_\infty^k\).
All three terms are multiplied by the propagation factor \(\gamma_k\),
which also depends on \(P\).

We next relate \(\mu_k\) and \(\nu_k\) in
\eqref{eq:source-path-bounds} to the accuracy of the solvers.
For sufficiently smooth problems, Runge--Kutta methods approximate
the solution and its derivative with respect to the initial value with the
same local order~\cite[Sec.~3]{gimeno2023variational}.
For such solvers with coarse order \(p\ge1\) and fine order at least \(p\),
comparison with the same exact local solution map shows that
\(\lVert\mathrm D\mathcal D_n\rVert_2=O(h^{p+1})\).
When this bound holds uniformly along \(\mathbf z_n^k(\tau)\) for
\(n\in\mathcal A_k\) and \(0\le\tau\le1\) with a constant
\(C_{\mathcal D}\) independent of \(h\), one may take
\(\mu_k=C_{\mathcal D}h^{p+1}\) in \eqref{eq:source-path-bounds}.
The constant \(\nu_k\) in~\eqref{eq:source-path-bounds} bounds the change
of \(\mathrm D\mathcal D_n\) between \(P\mathbf z_n^k(\tau)\) and
\(\mathbf z_n^k(\tau)\), so it involves
the second derivative of \(\mathcal D_n\), which the bound on
\(\mathrm D\mathcal D_n\) does not control.
For Runge--Kutta methods, the second derivative with respect to the initial
value is obtained by applying the same method to the second-order variational
equations~\cite[Thm.~3.1]{gimeno2023variational}.
The same comparison shows that
\(\lVert\mathrm D^2\mathcal D_n\rVert_2=O(h^{p+1})\), with the norm of
bilinear maps induced by \(\lVert\cdot\rVert_2\).
If each \(\mathcal D_n\) with \(n\in\mathcal A_k\) is \(C^2\)
on an open subset of \(\Omega_n\)
containing the segments joining \(P\mathbf z_n^k(\tau)\) and
\(\mathbf z_n^k(\tau)\) for \(0\le\tau\le1\), the mean value inequality yields
\[
  \bigl\lVert\mathrm D\mathcal D_n(\mathbf z_n^k(\tau))
     -\mathrm D\mathcal D_n(P\mathbf z_n^k(\tau))\bigr\rVert_2
  \le\max_{0\le s\le1}
     \bigl\lVert\mathrm D^2\mathcal D_n\bigl(P\mathbf z_n^k(\tau)
       +sQ\mathbf z_n^k(\tau)\bigr)\bigr\rVert_2
     \lVert Q\mathbf z_n^k(\tau)\rVert_2,
\]
for \(n\in\mathcal A_k\) and \(0\le\tau\le1\).
The second inequality in~\eqref{eq:source-path-bounds} holds with
\(\nu_k\) chosen as the maximum of the second-derivative norms above over
\(n\in\mathcal A_k\) and \(0\le\tau,s\le1\).
Hence \(\nu_k\) can be chosen as \(O(h^{p+1})\) when the bound on
\(\mathrm D^2\mathcal D_n\) is uniform on all these segments with a constant
independent of \(h\).
The following corollary applies
Theorem~\ref{lem:projection-dependent-one-sweep} to particular projections
with \(\mu_k=C_{\mathcal D}h^{p+1}\) and does not require an order for \(\nu_k\).

\begin{corollary}
  \label{cor:one-step-special-cases}
  Suppose that the assumptions of Theorem~\ref{lem:projection-dependent-one-sweep}
  hold with \(\mu_k=C_{\mathcal D}h^{p+1}\), where \(p\ge1\) is an integer
  and \(C_{\mathcal D}\ge0\) is a constant.

  \ivpresultpart\label{cor:full-projection} If \(P=I\), then
  \(\mathbf U^{k+1}=\mathbf u^\star\) and \(E_\infty^{k+1}=0\).

  \ivpresultpart\label{cor:zero-projection} If \(P=0\), then
  \[
    E_\infty^{k+1}
    \le\gamma_k C_{\mathcal D}h^{p+1}
      \max_{j\in\mathcal A_k}\lVert\mathbf e_j^k\rVert_2
    \le\gamma_k C_{\mathcal D}h^{p+1}E_\infty^k.
  \]

  \ivpresultpart\label{cor:old-state-projection}
  If \(P\mathbf U_n^k=\mathbf U_n^k\) for all
  \(n\in\mathcal A_k\), then
  \[
    E_\infty^{k+1}
    \le\gamma_k\theta_Q^k(P)\Bigl(C_{\mathcal D}h^{p+1}
      +\frac{\nu_k}{2}E_\infty^k\Bigr)E_\infty^k.
  \]
\end{corollary}
\begin{proof}
  For \(P=I\), \(Q=0\), so \(\theta_Q^k(I)=U_k^\perp(I)=0\).
  Hence \(\rho_k(I)=0\), and
  \eqref{eq:nonlinear-one-sweep-estimate} implies
  Corollary~\ref{cor:full-projection}.

  For \(P=0\), \(Q=I\), so \(\theta_P^k(0)=0\) and
  \(\theta_Q^k(0)E_\infty^k=\max_{j\in\mathcal A_k}\lVert\mathbf e_j^k\rVert_2
  \le E_\infty^k\).
  With \(\mu_k=C_{\mathcal D}h^{p+1}\), \eqref{eq:nonlinear-one-sweep-factor}
  reduces to \(\rho_k(0)=\gamma_kC_{\mathcal D}h^{p+1}\theta_Q^k(0)\).
  By \eqref{eq:nonlinear-one-sweep-estimate} and the identity
  \(\theta_Q^k(0)E_\infty^k
    =\max_{j\in\mathcal A_k}\lVert\mathbf e_j^k\rVert_2\),
  \[
    \begin{aligned}
      E_\infty^{k+1}&\le\rho_k(0)E_\infty^k\\
      &=\gamma_kC_{\mathcal D}h^{p+1}
        \max_{j\in\mathcal A_k}\lVert\mathbf e_j^k\rVert_2\\
      &\le\gamma_kC_{\mathcal D}h^{p+1}E_\infty^k.
    \end{aligned}
  \]
  This proves Corollary~\ref{cor:zero-projection}.

  For Corollary~\ref{cor:old-state-projection}, the identities
  \(Q\mathbf U_n^k=\mathbf0\) for \(n\in\mathcal A_k\) imply
  \(U_k^\perp(P)=0\) and \(Q\mathbf e_j^k=Q\mathbf u_j^\star\) for
  \(j\in\mathcal A_k\).
  Hence
  \[
    \theta_Q^k(P)E_\infty^k
    =\max_{j\in\mathcal A_k}\lVert Q\mathbf u_j^\star\rVert_2.
  \]
  With \(\mu_k=C_{\mathcal D}h^{p+1}\) and \(U_k^\perp(P)=0\),
  \eqref{eq:nonlinear-one-sweep-factor} reduces to
  \[
    \begin{aligned}
      \rho_k(P)
      &=\gamma_k\theta_Q^k(P)\Bigl(C_{\mathcal D}h^{p+1}
         +\frac{\nu_k}{2}\theta_P^k(P)E_\infty^k\Bigr)\\
      &\le\gamma_k\theta_Q^k(P)\Bigl(C_{\mathcal D}h^{p+1}
         +\frac{\nu_k}{2}E_\infty^k\Bigr),
    \end{aligned}
  \]
  where the inequality uses \(\theta_P^k(P)\le1\) and
  \(\gamma_k\theta_Q^k(P)\nu_kE_\infty^k\ge0\).
  Theorem~\ref{lem:projection-dependent-one-sweep} then proves
  Corollary~\ref{cor:old-state-projection}.
  By the identity for \(\theta_Q^k(P)E_\infty^k\), the right-hand side of the
  estimate in Corollary~\ref{cor:old-state-projection} equals
  \(\gamma_k\bigl(C_{\mathcal D}h^{p+1}+\tfrac{\nu_k}{2}E_\infty^k\bigr)
  \max_{j\in\mathcal A_k}\lVert Q\mathbf u_j^\star\rVert_2\).
\end{proof}

Theorem~\ref{lem:projection-dependent-one-sweep} provides a projection-dependent
convergence estimate with the factor \(\rho_k(P)\) in
\eqref{eq:nonlinear-one-sweep-factor}.
For \(P^k=0\), the map \(\mathcal H_{0,n}\) differs from \(\mathcal G_n\) by
the constant \(\mathcal D_n(\mathbf0)\), so scheme
\eqref{eq:projection-parareal} is the classical iteration
\eqref{eq:classical-parareal}.
For \(P=0\), Corollary~\ref{cor:zero-projection} gives
an estimate for classical nonlinear \emph{parareal} with the same separation
of error generation and propagation as in~\cite[Thm.~1 and its proof]{gander2008nonlinear}.
A related factorization of the nonlinear iteration Jacobian is given
in~\cite[Thm.~4.1 and Eq.~(4.4)]{antonucci2026sources}.

\section{Asymptotic error analysis}
\label{sec:form-boundary}

Section~\ref{sec:projection-dependent-analysis} fixes \(h\) and
analyzes the error as the iteration proceeds.
This section analyzes the errors produced by one step of schemes
\eqref{eq:direct-split-iteration} and \eqref{eq:projection-parareal}
as \(h\to0\).
Let \(T>0\) and \(h_0>0\) be fixed.
Throughout this section, we consider a family of uniform grids on
\([0,T]\) along which \(h\to0\).
Each grid of the family has the form
\begin{equation}
  \label{eq:uniform-grid-family}
  T_n=nh\ (n=0,\ldots,N),\qquad h=\frac TN,\qquad 0<h\le h_0,
\end{equation}
with \(N\ge2\) an integer.
On each grid of the family in \eqref{eq:uniform-grid-family},
the fine reference solution \(\mathbf u^\star\) in
\eqref{eq:fine-reference} is assumed well defined.
The indices, projections, approximations, and derived quantities may depend
on the grid.
This dependence is suppressed in the notation.

Let \(P\) be an orthogonal projection and set \(Q=I-P\).
For an admissible step of scheme \eqref{eq:projection-parareal}
with \(P^k=P\), the subtraction used in
\eqref{eq:corrected-error-identity} yields
\[
  \begin{aligned}
    \mathbf e_{n+1}^{k+1}
    &=\mathcal H_{P,n}(\mathbf u_n^\star)
      -\mathcal H_{P,n}(\mathbf U_n^{k+1})+\mathcal R_{P,n}(\mathbf u_n^\star)
      -\mathcal R_{P,n}(\mathbf U_n^k)
  \end{aligned}
\]
for \(0\le n<N\) such that \(P\mathbf u_n^\star\in\Omega_n\).
In this identity, the difference
\(\mathcal H_{P,n}(\mathbf u_n^\star)-\mathcal H_{P,n}(\mathbf U_n^{k+1})\)
propagates the new error \(\mathbf e_n^{k+1}\) to \(T_{n+1}\).
The remaining term
\(\mathcal R_{P,n}(\mathbf u_n^\star)-\mathcal R_{P,n}(\mathbf U_n^k)\)
gives the contribution generated by the old error \(\mathbf e_n^k\).
We estimate these differences to bound
\(\lVert\mathbf e_{n+1}^{k+1}\rVert_2\) in terms of \(h\).

We next collect the notation used in the analysis.
For \(\mathbf y,P\mathbf y\in\Omega_n\), the residual
\eqref{eq:unresolved-residual} of scheme \eqref{eq:projection-parareal} satisfies
\[
  \mathcal R_{P,n}(\mathbf y)
  =\mathcal F_n(\mathbf y)-\mathcal H_{P,n}(\mathbf y)
  =\mathcal D_n(\mathbf y)-\mathcal D_n(P\mathbf y).
\]
Here \(\mathcal D_n=\mathcal F_n-\mathcal G_n\).
For scheme \eqref{eq:direct-split-iteration}, the residual is
\begin{equation}
  \label{eq:s-residual}
  \begin{aligned}
  \mathcal R^{\mathcal S}_{P,n}(\mathbf y)
  &:=\mathcal F_n(\mathbf y)-\mathcal S_{P,n}(\mathbf y)\\
  &=\mathcal R_{P,n}(\mathbf y)
    +\mathcal G_n(\mathbf y)-\mathcal G_n(P\mathbf y)-\mathcal G_n(Q\mathbf y),
  \qquad \mathbf y,P\mathbf y,Q\mathbf y\in\Omega_n,
  \end{aligned}
\end{equation}
by \eqref{eq:exact-form-boundary}.
Define \(\mathbf r_n\) on \(\Omega_n\) for \(0\le n<N\) by
\[
  \mathcal G_n(\mathbf y)
  =\mathbf y+h\mathbf f(\mathbf y,T_n)+\mathbf r_n(\mathbf y).
\]
For \(\mathbf y\in\mathbb R^d\), define the mixed difference of
\(\mathbf f(\cdot,T_n)\) at \(\mathbf0\) with increments \(P\mathbf y\) and
\(Q\mathbf y\) by
\begin{equation}
  \label{eq:mixed-interaction}
  \boldsymbol\delta_n(\mathbf y;P)
  :=\mathbf f(\mathbf y,T_n)-\mathbf f(P\mathbf y,T_n)
    -\mathbf f(Q\mathbf y,T_n)+\mathbf f(\mathbf0,T_n).
\end{equation}
\Needspace{4\baselineskip}
We impose the following assumption at one pair \((\mathbf x,\mathbf w)\).
\begin{assumption}
  \label{ass:local-form-source-regularity}
  Fix a quadruple \((n,P;\mathbf x,\mathbf w)\), where
  \(0\le n\le N-1\) and \(P\) is an orthogonal projection, and let
  \[
    \Sigma:=\{(1-\tau)\mathbf w+\tau\mathbf x:0\le\tau\le1\}.
  \]
  Let \(p\ge1\) be an integer and \(C_{\mathcal D},C_r,L_\delta\ge0\)
  constants.

  \ivpassumptionpart\label{ass:form-derivative-bound}
  \(\mathcal D_n\) is \(C^1\) on an open subset of
  \(\Omega_n\) containing \(\Sigma\cup P\Sigma\), and
  \begin{equation}
    \label{eq:derivative-source-bounds}
    \sup_{\mathbf y\in\Sigma\cup P\Sigma}
    \lVert\mathrm D\mathcal D_n(\mathbf y)\rVert_2
    \le C_{\mathcal D}h^{p+1}.
  \end{equation}

  \ivpassumptionpart\label{ass:form-expansion-bounds}
  Assumption~\ref{ass:form-derivative-bound} holds.
  On an open subset of \(\Omega_n\) containing
  \(\Sigma\cup P\Sigma\cup Q\Sigma\), \(\mathbf r_n\) is \(C^1\) with
  \(\lVert\mathrm D\mathbf r_n(\mathbf y)\rVert_2\le C_rh^2\)
  for \(\mathbf y\in\Sigma\cup P\Sigma\cup Q\Sigma\).
  On an open neighborhood of \(\Sigma\) in \(\mathbb R^d\),
  \(\boldsymbol\delta_n(\cdot;P)\) is \(C^1\) with
  \(\lVert\mathrm D\boldsymbol\delta_n(\mathbf y;P)\rVert_2\le L_\delta\)
  for \(\mathbf y\in\Sigma\).
\end{assumption}

Uniformity of Assumption~\ref{ass:local-form-source-regularity} over the
family in \eqref{eq:uniform-grid-family} is specified in
Section~\ref{subsec:corrected-iteration-form-boundary}.
The discrepancy estimate of Gander and
Hairer~\cite[Sec.~3, Eq.~(7), and Thm.~1]{gander2008nonlinear}
assumes an exact fine solver.
Here \eqref{eq:derivative-source-bounds} is imposed directly over
\(\Sigma\cup P\Sigma\) on the discrepancy
\(\mathcal D_n=\mathcal F_n-\mathcal G_n\) of the two numerical solvers.

\subsection{Lipschitz bounds for two residuals}
\label{subsec:exact-form-boundary}

To compare \(\mathcal R_{P,n}\) in
\eqref{eq:unresolved-residual} with \(\mathcal R^{\mathcal S}_{P,n}\) in
\eqref{eq:s-residual},
we consider the increment of each residual from \(\mathbf w\) to \(\mathbf x\).
For \(\mathbf x,P\mathbf x,\mathbf w,P\mathbf w\in\Omega_n\), define
\begin{equation}
  \label{eq:two-state-source}
  \mathbf s_{\mathcal H,n}(\mathbf x,\mathbf w)
  :=\mathcal R_{P,n}(\mathbf x)-\mathcal R_{P,n}(\mathbf w).
\end{equation}
For \(\mathbf x,P\mathbf x,Q\mathbf x,\mathbf w,P\mathbf w,Q\mathbf w\in\Omega_n\),
define
\[
  \mathbf s_{\mathcal S,n}(\mathbf x,\mathbf w)
  :=\mathcal R^{\mathcal S}_{P,n}(\mathbf x)-\mathcal R^{\mathcal S}_{P,n}(\mathbf w).
\]
The dependence of \(\mathbf s_{\mathcal H,n}\) and \(\mathbf s_{\mathcal S,n}\) on \(P\) is
suppressed.
For the iteration step analyzed in Section~\ref{subsec:projection-dependent-one-sweep}
under Assumption~\ref{ass:local-pathwise-regularity}, set \(P=P^k\).
With \(\mathcal A_k\) as in \eqref{eq:active-set},
definition \eqref{eq:old-state-source} gives
\[
  \mathbf s_{\mathcal H,n}(\mathbf u_n^\star,\mathbf U_n^k)
  =\mathbf s_{n+1}^k,\qquad n\in\mathcal A_k.
\]
The following lemma bounds
\(\lVert\mathbf s_{\mathcal H,n}(\mathbf x,\mathbf w)\rVert_2\) and
\(\lVert\mathbf s_{\mathcal S,n}(\mathbf x,\mathbf w)\rVert_2\) by multiples of
\(\lVert\mathbf x-\mathbf w\rVert_2\).
\begin{lemma}
  \label{lem:form-specific-source}
  Let \(0\le n\le N-1\), and let \(P\) be an orthogonal projection.
  Fix a pair \((\mathbf x,\mathbf w)\).
  Under Assumption~\textup{\ref{ass:form-derivative-bound}} at
  \((n,P;\mathbf x,\mathbf w)\),
  \begin{equation}
    \label{eq:source-bound-discrepancy}
    \lVert\mathbf s_{\mathcal H,n}(\mathbf x,\mathbf w)\rVert_2
    \le2C_{\mathcal D}h^{p+1}\lVert\mathbf x-\mathbf w\rVert_2.
  \end{equation}
  Under Assumption~\textup{\ref{ass:form-expansion-bounds}}
  at the same quadruple,
  \begin{equation}
    \label{eq:source-bound-split}
    \lVert\mathbf s_{\mathcal S,n}(\mathbf x,\mathbf w)\rVert_2
    \le(2C_{\mathcal D}h^{p+1}+L_\delta h+3C_rh^2)
          \lVert\mathbf x-\mathbf w\rVert_2.
  \end{equation}

\end{lemma}
\begin{proof}
  Write \(\mathbf v:=\mathbf x-\mathbf w\).
  By the fundamental theorem of calculus and \eqref{eq:unresolved-residual},
  we have
  \[
    \mathbf s_{\mathcal H,n}(\mathbf x,\mathbf w)
    =\int_0^1\mathrm D\mathcal D_n(\mathbf w+\tau\mathbf v)\mathbf v\,d\tau
     -\int_0^1\mathrm D\mathcal D_n(P\mathbf w+\tau P\mathbf v)P\mathbf v\,d\tau.
  \]
  For \(0\le\tau\le1\), both \(\mathbf w+\tau\mathbf v\) and
  \(P\mathbf w+\tau P\mathbf v\) lie in \(\Sigma\cup P\Sigma\), so
  \eqref{eq:derivative-source-bounds} and \(\lVert P\rVert_2\le1\) imply
  \eqref{eq:source-bound-discrepancy}.

  Under Assumption~\ref{ass:form-expansion-bounds},
  the definitions of
  \(\mathcal R_{P,n}\) in \eqref{eq:unresolved-residual} and
  \(\mathcal R^{\mathcal S}_{P,n}\) in \eqref{eq:s-residual} yield
  \[
    \begin{aligned}
      \mathcal R^{\mathcal S}_{P,n}(\mathbf y)-\mathcal R_{P,n}(\mathbf y)
      &=[\mathcal F_n(\mathbf y)-\mathcal S_{P,n}(\mathbf y)]
        -[\mathcal F_n(\mathbf y)-\mathcal H_{P,n}(\mathbf y)]\\
      &=\mathcal H_{P,n}(\mathbf y)-\mathcal S_{P,n}(\mathbf y)
    \end{aligned}
  \]
  for \(\mathbf y,P\mathbf y,Q\mathbf y\in\Omega_n\).
  Subtracting this identity at \(\mathbf w\) from the identity at
  \(\mathbf x\), we obtain
  \begin{equation}
    \label{eq:exact-source-boundary}
    \mathbf s_{\mathcal S,n}(\mathbf x,\mathbf w)-\mathbf s_{\mathcal H,n}(\mathbf x,\mathbf w)
    =[\mathcal H_{P,n}(\mathbf x)-\mathcal S_{P,n}(\mathbf x)]
     -[\mathcal H_{P,n}(\mathbf w)-\mathcal S_{P,n}(\mathbf w)].
  \end{equation}
  For \(\mathbf y\in\Omega_n\), the definition of \(\mathbf r_n\) gives
  \[
    \mathcal G_n(\mathbf y)
    =\mathbf y+h\mathbf f(\mathbf y,T_n)+\mathbf r_n(\mathbf y).
  \]
  Substituting this identity at
  \(\mathbf y,P\mathbf y,Q\mathbf y\in\Omega_n\) into
  \eqref{eq:exact-form-boundary} and using
  \(\mathbf y-P\mathbf y-Q\mathbf y=\mathbf0\) and
  \eqref{eq:mixed-interaction} yields
  \begin{equation}
    \label{eq:interaction-expansion}
    \begin{aligned}
    \mathcal H_{P,n}(\mathbf y)-\mathcal S_{P,n}(\mathbf y)
    &=h[\boldsymbol\delta_n(\mathbf y;P)-\mathbf f(\mathbf0,T_n)]+\mathbf r_n(\mathbf y)-\mathbf r_n(P\mathbf y)-\mathbf r_n(Q\mathbf y).
    \end{aligned}
  \end{equation}
  The constant term \(-h\mathbf f(\mathbf0,T_n)\) cancels
  when the increment from \(\mathbf w\) to \(\mathbf x\) is taken.
  Substituting \eqref{eq:interaction-expansion} into
  \eqref{eq:exact-source-boundary} gives
  \begin{equation}
    \label{eq:source-boundary-expansion}
    \mathbf s_{\mathcal S,n}(\mathbf x,\mathbf w)-\mathbf s_{\mathcal H,n}(\mathbf x,\mathbf w)
    =h[\boldsymbol\delta_n(\mathbf x;P)-\boldsymbol\delta_n(\mathbf w;P)]
       +\boldsymbol\rho_n(\mathbf x,\mathbf w),
  \end{equation}
  where
  \begin{equation}
    \label{eq:rho-definition}
    \begin{aligned}
    \boldsymbol\rho_n(\mathbf x,\mathbf w)
    &:=[\mathbf r_n(\mathbf x)-\mathbf r_n(\mathbf w)]-[\mathbf r_n(P\mathbf x)-\mathbf r_n(P\mathbf w)]-[\mathbf r_n(Q\mathbf x)-\mathbf r_n(Q\mathbf w)].
    \end{aligned}
  \end{equation}
  Applying the fundamental theorem of calculus to the three
  differences of \(\mathbf r_n\) in \eqref{eq:rho-definition} and using
  Assumption~\ref{ass:form-expansion-bounds}
  and \(\lVert P\rVert_2,\lVert Q\rVert_2\le1\), we obtain
  \begin{equation}
    \label{eq:rho-bound}
    \begin{aligned}
      \lVert\boldsymbol\rho_n(\mathbf x,\mathbf w)\rVert_2
      &\le C_rh^2\bigl(\lVert\mathbf v\rVert_2
         +\lVert P\mathbf v\rVert_2+\lVert Q\mathbf v\rVert_2\bigr)\\
      &\le3C_rh^2\lVert\mathbf v\rVert_2.
    \end{aligned}
  \end{equation}
  Applying the same argument to \(\boldsymbol\delta_n(\cdot;P)\)
  along \(\Sigma\) gives
  \begin{equation}
    \label{eq:delta-increment-bound}
    \lVert\boldsymbol\delta_n(\mathbf x;P)-\boldsymbol\delta_n(\mathbf w;P)\rVert_2
    \le L_\delta\lVert\mathbf v\rVert_2.
  \end{equation}
  Applying the triangle inequality to
  \eqref{eq:source-boundary-expansion} and using \eqref{eq:rho-bound},
  \eqref{eq:delta-increment-bound}, and \eqref{eq:source-bound-discrepancy}
  yields \eqref{eq:source-bound-split}.
\end{proof}

In the bounds of Lemma~\ref{lem:form-specific-source}, the
increment of \(\mathcal R_{P,n}\) has the discrepancy order \(h^{p+1}\),
whereas that of \(\mathcal R^{\mathcal S}_{P,n}\) carries an additional
\(O(h)\) term from the mixed difference \(\boldsymbol\delta_n(\cdot;P)\).
This term accounts for the different powers of \(h\) in the estimates for
schemes \eqref{eq:direct-split-iteration} and
\eqref{eq:projection-parareal} in
Theorems~\ref{thm:first-active-form-boundary}
and~\ref{prop:later-node-form-envelopes}.

\subsection{Asymptotic error estimates for the two schemes}
\label{subsec:corrected-iteration-form-boundary}

We next estimate the errors after one step of each scheme from a common
approximation, using the same orthogonal projection.
On each grid of the family in \eqref{eq:uniform-grid-family},
choose \(0\le k\le N-2\), an orthogonal projection \(P\)
kept fixed across all coarse intervals of the grid, and a
common old approximation \(\mathbf U^k\) satisfying
\[
  \mathbf U_j^k=\mathbf u_j^\star,\qquad 0\le j\le k.
\]

From the common old approximation \(\mathbf U^k\), take one admissible step
of scheme \eqref{eq:projection-parareal} and one of
scheme \eqref{eq:direct-split-iteration} to obtain
\(\mathbf U_{\mathcal H}^{k+1}\) and \(\mathbf U_{\mathcal S}^{k+1}\),
respectively.
For \(0\le j\le N\), set
\[
  \mathbf e_{\mathcal H,j}^{k+1}:=\mathbf u_j^\star-\mathbf U_{\mathcal H,j}^{k+1},
  \qquad
  \mathbf e_{\mathcal S,j}^{k+1}:=\mathbf u_j^\star-\mathbf U_{\mathcal S,j}^{k+1}.
\]
By Lemma~\ref{lem:finite-step-exactness}, we have
\[
  \mathbf e_{\mathcal H,j}^{k+1}=\mathbf e_{\mathcal S,j}^{k+1}=\mathbf0,
  \qquad 0\le j\le k+1.
\]

For each \(j=0,\ldots,N-1\) such that
\(\mathbf u_j^\star,P\mathbf u_j^\star,Q\mathbf u_j^\star\in\Omega_j\),
subtracting schemes \eqref{eq:projection-parareal} and \eqref{eq:direct-split-iteration} from \eqref{eq:fine-reference} gives
\begin{equation}
  \label{eq:two-branch-error-recurrence}
  \begin{aligned}
  \mathbf e_{\mathcal H,j+1}^{k+1}
  &=\mathcal H_{P,j}(\mathbf u_j^\star)
    -\mathcal H_{P,j}(\mathbf U_{\mathcal H,j}^{k+1})
    +\mathbf s_{\mathcal H,j}(\mathbf u_j^\star,\mathbf U_j^k),\\
  \mathbf e_{\mathcal S,j+1}^{k+1}
  &=\mathcal S_{P,j}(\mathbf u_j^\star)
    -\mathcal S_{P,j}(\mathbf U_{\mathcal S,j}^{k+1})
    +\mathbf s_{\mathcal S,j}(\mathbf u_j^\star,\mathbf U_j^k).
  \end{aligned}
\end{equation}
On each line of \eqref{eq:two-branch-error-recurrence}, the map
difference propagates the corresponding new error.
Both \(\mathbf s_{\mathcal H,j}\) and \(\mathbf s_{\mathcal S,j}\)
are evaluated at the common pair
\((\mathbf u_j^\star,\mathbf U_j^k)\).
At \(j=k+1\), both map differences vanish, so
\[
  \begin{aligned}
  \mathbf e_{\mathcal H,k+2}^{k+1}
  &=\mathbf s_{\mathcal H,k+1}(\mathbf u_{k+1}^\star,\mathbf U_{k+1}^k),\\
  \mathbf e_{\mathcal S,k+2}^{k+1}
  &=\mathbf s_{\mathcal S,k+1}(\mathbf u_{k+1}^\star,\mathbf U_{k+1}^k).
  \end{aligned}
\]
At subsequent time points, the differences involving
\(\mathcal H_{P,j}\) and \(\mathcal S_{P,j}\) in
\eqref{eq:two-branch-error-recurrence} are evaluated at the respective pairs
\((\mathbf u_j^\star,\mathbf U_{\mathcal H,j}^{k+1})\) and
\((\mathbf u_j^\star,\mathbf U_{\mathcal S,j}^{k+1})\).
These two new approximations need not coincide.

On each grid of the family in \eqref{eq:uniform-grid-family},
prescribe a collection of quadruples \((n,P;\mathbf x,\mathbf w)\).
We say that Assumption~\ref{ass:form-expansion-bounds} holds uniformly on these
collections if there exist an integer \(p\ge1\) and finite constants
\(C_{\mathcal D},C_r,L_\delta\ge0\) such that it holds at every
prescribed quadruple on every grid with the same parameters.
The open sets in
Assumption~\ref{ass:local-form-source-regularity} may depend on both the grid
and the quadruple.

To obtain a lower bound for \(\lVert\mathbf e_{\mathcal S,k+2}^{k+1}\rVert_2\),
we impose the following assumption.
\begin{assumption}
  \label{ass:first-active-source-separation}
  On each grid of the family in \eqref{eq:uniform-grid-family},
  Assumption~\ref{ass:form-expansion-bounds}
  holds uniformly at
  \[
    (k+1,P;\mathbf u_{k+1}^\star,\mathbf U_{k+1}^k).
  \]
  The old error satisfies \(\mathbf e_{k+1}^k\ne\mathbf0\), and
  \[
    \Delta\boldsymbol\delta
    :=\boldsymbol\delta_{k+1}(\mathbf u_{k+1}^\star;P)
     -\boldsymbol\delta_{k+1}(\mathbf U_{k+1}^k;P)
  \]
  obeys
  \begin{equation}
    \label{eq:first-active-source-separation}
    \lVert\Delta\boldsymbol\delta\rVert_2\ge c_\delta\lVert\mathbf e_{k+1}^k\rVert_2
  \end{equation}
  for a constant \(c_\delta>0\) independent of the grid.
\end{assumption}

The lower bound is required only for the pair
\((\mathbf u_{k+1}^{\star},\mathbf U_{k+1}^{k})\) on each grid of the family,
rather than for all pairs. Nonlinearity of \(\mathbf f\) alone does not
guarantee a positive constant \(c_\delta\) independent of the grid.
The following theorem compares
\(\lVert\mathbf e_{\mathcal H,k+2}^{k+1}\rVert_2\) and
\(\lVert\mathbf e_{\mathcal S,k+2}^{k+1}\rVert_2\) at the first potentially
inexact time point \(T_{k+2}\).
\begin{theorem}
  \label{thm:first-active-form-boundary}
  Consider a family of uniform grids along which \(h\to0\),
  satisfying \eqref{eq:uniform-grid-family} with fixed \(T,h_0>0\) and well-defined fine reference solutions
  \(\mathbf u^\star\) from \eqref{eq:fine-reference}.
  On each grid, let \(0\le k\le N-2\) and let \(\mathbf U^k\) satisfy
  \(\mathbf U_j^k=\mathbf u_j^\star\) for \(0\le j\le k\).
  Take one admissible step of scheme \eqref{eq:projection-parareal} and one of
  scheme \eqref{eq:direct-split-iteration} from \(\mathbf U^k\) with the same
  orthogonal projection \(P\).
  If Assumption~\ref{ass:first-active-source-separation} holds, then the
  errors at the first potentially inexact time point \(T_{k+2}\) satisfy
  \begin{equation}
    \label{eq:first-active-scales}
    \frac{\lVert\mathbf e_{\mathcal H,k+2}^{k+1}\rVert_2}
         {\lVert\mathbf e_{k+1}^k\rVert_2}=O(h^{p+1}),
    \qquad
    \frac{\lVert\mathbf e_{\mathcal S,k+2}^{k+1}\rVert_2}
         {\lVert\mathbf e_{k+1}^k\rVert_2}=\Theta(h),
    \qquad
    \frac{\lVert\mathbf e_{\mathcal H,k+2}^{k+1}\rVert_2}
         {\lVert\mathbf e_{\mathcal S,k+2}^{k+1}\rVert_2}=O(h^p).
  \end{equation}
  The constants in the \(O(\cdot)\) and \(\Theta(\cdot)\) bounds
  are independent of the grid.
\end{theorem}
\begin{proof}
  On each grid of the family in \eqref{eq:uniform-grid-family},
  Lemma~\ref{lem:finite-step-exactness} implies
  \[
    \mathbf U_{\mathcal H,k+1}^{k+1}
    =\mathbf U_{\mathcal S,k+1}^{k+1}
    =\mathbf u_{k+1}^\star.
  \]
  Admissibility of the step of scheme \eqref{eq:direct-split-iteration} therefore ensures that
  \(\mathbf u_{k+1}^\star,P\mathbf u_{k+1}^\star,
    Q\mathbf u_{k+1}^\star\in\Omega_{k+1}\).
  The map differences in \eqref{eq:two-branch-error-recurrence}
  vanish at \(j=k+1\).
  With \(\boldsymbol\rho:=\boldsymbol\rho_{k+1}(\mathbf u_{k+1}^\star,\mathbf U_{k+1}^k)\),
  \eqref{eq:two-branch-error-recurrence} at \(j=k+1\) and
  \eqref{eq:source-boundary-expansion} yield
  \begin{equation}
    \label{eq:first-active-form-error-boundary}
    \mathbf e_{\mathcal H,k+2}^{k+1}
    =\mathbf s_{\mathcal H,k+1}(\mathbf u_{k+1}^\star,\mathbf U_{k+1}^k),
    \qquad
    \mathbf e_{\mathcal S,k+2}^{k+1}
    =\mathbf e_{\mathcal H,k+2}^{k+1}+h\Delta\boldsymbol\delta+\boldsymbol\rho.
  \end{equation}
  By Lemma~\ref{lem:form-specific-source} and the bounds
  \eqref{eq:rho-bound} and \eqref{eq:delta-increment-bound}, we have
  \[
    \begin{aligned}
    \lVert\mathbf e_{\mathcal H,k+2}^{k+1}\rVert_2
    &\le2C_{\mathcal D}h^{p+1}\lVert\mathbf e_{k+1}^k\rVert_2,\\
    \lVert\boldsymbol\rho\rVert_2&\le3C_rh^2\lVert\mathbf e_{k+1}^k\rVert_2,\\
    \lVert\Delta\boldsymbol\delta\rVert_2&\le L_\delta\lVert\mathbf e_{k+1}^k\rVert_2.
    \end{aligned}
  \]
  Assumption~\ref{ass:first-active-source-separation} gives
  \(\mathbf e_{k+1}^k\ne\mathbf0\) and
  \(\lVert\Delta\boldsymbol\delta\rVert_2\ge c_\delta\lVert\mathbf e_{k+1}^k\rVert_2\).
  Applying the triangle and reverse triangle inequalities to
  \eqref{eq:first-active-form-error-boundary}, we then obtain
  \[
    c_\delta h-3C_rh^2-2C_{\mathcal D}h^{p+1}
    \le\frac{\lVert\mathbf e_{\mathcal S,k+2}^{k+1}\rVert_2}
             {\lVert\mathbf e_{k+1}^k\rVert_2}
    \le L_\delta h+3C_rh^2+2C_{\mathcal D}h^{p+1}.
  \]
  The exponent \(p\) and the constants \(c_\delta\), \(L_\delta\), \(C_r\),
  \(C_{\mathcal D}\), and \(h_0\) are independent of the grid.
  We can thus choose a
  grid-independent \(h_1\in(0,h_0]\) so that
  \(3C_rh+2C_{\mathcal D}h^p\le c_\delta/2\) for \(0<h\le h_1\).
  For such \(h\), the higher-order terms are at most half of the leading
  term \(c_\delta h\).
  For \(0<h\le h_1\), the preceding estimate implies
  \[
    \frac{c_\delta}{2}h
    \le\frac{\lVert\mathbf e_{\mathcal S,k+2}^{k+1}\rVert_2}
             {\lVert\mathbf e_{k+1}^k\rVert_2}
    \le\left(L_\delta+\frac{c_\delta}{2}\right)h.
  \]
  Combining the bound on
  \(\lVert\mathbf e_{\mathcal H,k+2}^{k+1}\rVert_2\) above with the lower bound
  in the preceding display yields
  \[
    \frac{\lVert\mathbf e_{\mathcal H,k+2}^{k+1}\rVert_2}
         {\lVert\mathbf e_{\mathcal S,k+2}^{k+1}\rVert_2}
    \le\frac{4C_{\mathcal D}}{c_\delta}h^p.
  \]
  Thus \eqref{eq:first-active-scales} holds for \(0<h\le h_1\) with the
  grid-independent constants \(2C_{\mathcal D}\), \(c_\delta/2\),
  \(L_\delta+c_\delta/2\), and \(4C_{\mathcal D}/c_\delta\).
\end{proof}

Even without the lower bound \eqref{eq:first-active-source-separation},
the identities in \eqref{eq:first-active-form-error-boundary},
Lemma~\ref{lem:form-specific-source}, and the bound \eqref{eq:rho-bound} imply
\[
  \bigl\lVert\mathbf e_{\mathcal S,k+2}^{k+1}-h\Delta\boldsymbol\delta\bigr\rVert_2
  \le\bigl(2C_{\mathcal D}h^{p+1}+3C_rh^2\bigr)\lVert\mathbf e_{k+1}^k\rVert_2 .
\]
The two-sided \(\Theta(h)\) estimate and the bound on the error ratio in
\eqref{eq:first-active-scales} additionally use the lower bound.
The lower bound \eqref{eq:first-active-source-separation} in
Assumption~\ref{ass:first-active-source-separation} applies to
\[
  \Delta\boldsymbol\delta
  =\boldsymbol\delta_{k+1}(\mathbf u_{k+1}^\star;P)
   -\boldsymbol\delta_{k+1}(\mathbf U_{k+1}^k;P),
\]
rather than to the two values separately.
This lower bound cannot hold for \(P=0\), \(P=I\), or linear
\(\mathbf f(\cdot,T_{k+1})\), because
\(\boldsymbol\delta_{k+1}(\cdot;P)\) vanishes in each case.

On each grid of the family in \eqref{eq:uniform-grid-family},
choose \(k+1\le m\le N-1\) and take \(T_{m+1}\) as the target time point.
We now apply \eqref{eq:two-branch-error-recurrence} up to this target.
In this recurrence,
\[
  \begin{aligned}
  \mathbf s_{\mathcal H,j}(\mathbf u_j^\star,\mathbf U_j^k)
  &=\mathcal R_{P,j}(\mathbf u_j^\star)-\mathcal R_{P,j}(\mathbf U_j^k),\\
  \mathbf s_{\mathcal S,j}(\mathbf u_j^\star,\mathbf U_j^k)
  &=\mathcal R^{\mathcal S}_{P,j}(\mathbf u_j^\star)
    -\mathcal R^{\mathcal S}_{P,j}(\mathbf U_j^k).
  \end{aligned}
\]
Here \(\mathcal R_{P,j}\) and \(\mathcal R^{\mathcal S}_{P,j}\)
are defined in \eqref{eq:unresolved-residual} and \eqref{eq:s-residual},
respectively.
Set \(q_h:=m-k\).
Thus \(q_h\) counts the coarse intervals from \(T_{k+1}\) to \(T_{m+1}\),
and \(q_hh=T_{m+1}-T_{k+1}\).
To bound the errors at \(T_{m+1}\), we impose the following
assumption on the intervals \([T_j,T_{j+1}]\), \(j=k+1,\ldots,m\).
\begin{assumption}
  \label{ass:active-prefix-source-transport}
  On each grid of the family in \eqref{eq:uniform-grid-family},
  Assumption~\ref{ass:form-expansion-bounds}
  holds uniformly at
  \[
    (j,P;\mathbf u_j^\star,\mathbf U_j^k),\qquad j=k+1,\ldots,m.
  \]
  For a constant \(L_{\mathrm{tr}}\ge0\) independent of the interval and the
  grid,
  \begin{equation}
    \label{eq:later-node-transport-stability}
    \begin{aligned}
    \lVert\mathcal H_{P,j}(\mathbf u_j^\star)
    -\mathcal H_{P,j}(\mathbf U_{\mathcal H,j}^{k+1})\rVert_2
    &\le(1+L_{\mathrm{tr}}h)\lVert\mathbf e_{\mathcal H,j}^{k+1}\rVert_2,\\
    \lVert\mathcal S_{P,j}(\mathbf u_j^\star)
    -\mathcal S_{P,j}(\mathbf U_{\mathcal S,j}^{k+1})\rVert_2
    &\le(1+L_{\mathrm{tr}}h)\lVert\mathbf e_{\mathcal S,j}^{k+1}\rVert_2,
    \end{aligned}
    \qquad j=k+1,\ldots,m.
  \end{equation}
\end{assumption}

The inequalities in \eqref{eq:later-node-transport-stability} are
pointwise analogues of the coarse-solver stability condition in
Gander and Hairer~\cite[Eq.~(8)]{gander2008nonlinear}.
Here they are required at the pairs
\((\mathbf u_j^\star,\mathbf U_{\mathcal H,j}^{k+1})\) and
\((\mathbf u_j^\star,\mathbf U_{\mathcal S,j}^{k+1})\).
At \(j=k+1\), both sides of each inequality vanish by
Lemma~\ref{lem:finite-step-exactness}.
The following theorem bounds
\(\lVert\mathbf e_{\mathcal H,m+1}^{k+1}\rVert_2\) and
\(\lVert\mathbf e_{\mathcal S,m+1}^{k+1}\rVert_2\) at the target time point
\(T_{m+1}\) by multiples of
\(\max_{k+1\le j\le m}\lVert\mathbf e_j^k\rVert_2\).
The powers of \(h\) in these bounds depend on how the number
\(q_h\) of coarse intervals varies as \(h\to0\).
\begin{theorem}
  \label{prop:later-node-form-envelopes}
  Consider a family of uniform grids along which \(h\to0\),
  satisfying \eqref{eq:uniform-grid-family} with fixed \(T,h_0>0\) and well-defined fine reference solutions
  \(\mathbf u^\star\) from \eqref{eq:fine-reference}.
  On each grid, let \(0\le k\le N-2\) and \(k+1\le m\le N-1\), and let
  \(\mathbf U^k\) satisfy \(\mathbf U_j^k=\mathbf u_j^\star\) for \(0\le j\le k\).
  Take one admissible step of scheme \eqref{eq:projection-parareal} and one of
  scheme \eqref{eq:direct-split-iteration} from \(\mathbf U^k\) with the same
  orthogonal projection \(P\).
  Suppose that Assumption~\ref{ass:active-prefix-source-transport} holds,
  and write \(q_h=m-k\).
  If \(q_h\le q_0\) for a fixed integer \(q_0\ge1\), then
  \begin{equation}
    \label{eq:fixed-offset-form-envelopes}
    \begin{aligned}
    \lVert\mathbf e_{\mathcal H,m+1}^{k+1}\rVert_2
    &=O(h^{p+1})\max_{k+1\le j\le m}\lVert\mathbf e_j^k\rVert_2,\\
    \lVert\mathbf e_{\mathcal S,m+1}^{k+1}\rVert_2
    &=O(h)\max_{k+1\le j\le m}\lVert\mathbf e_j^k\rVert_2.
    \end{aligned}
  \end{equation}
  If \(q_hh\le\tau_*\) for a fixed \(\tau_*>0\), then
  \begin{equation}
    \label{eq:physical-window-form-envelopes}
    \begin{aligned}
    \lVert\mathbf e_{\mathcal H,m+1}^{k+1}\rVert_2
    &=O(h^p)\max_{k+1\le j\le m}\lVert\mathbf e_j^k\rVert_2,\\
    \lVert\mathbf e_{\mathcal S,m+1}^{k+1}\rVert_2
    &=O(1)\max_{k+1\le j\le m}\lVert\mathbf e_j^k\rVert_2.
    \end{aligned}
  \end{equation}
  The constants in the \(O(\cdot)\) bounds are independent of the grid.
\end{theorem}
\begin{proof}
  On each grid of the family in \eqref{eq:uniform-grid-family},
  Assumption~\ref{ass:active-prefix-source-transport} implies
  \(\mathbf u_j^\star,P\mathbf u_j^\star,Q\mathbf u_j^\star\in\Omega_j\)
  for \(j=k+1,\ldots,m\).
  For \(j=k+1,\ldots,m\), applying Lemma~\ref{lem:form-specific-source} at
  \((j,P;\mathbf u_j^\star,\mathbf U_j^k)\) and
  \eqref{eq:later-node-transport-stability} to
  \eqref{eq:two-branch-error-recurrence} yields
  \[
    \begin{aligned}
    \lVert\mathbf e_{\mathcal H,j+1}^{k+1}\rVert_2
    &\le(1+L_{\mathrm{tr}}h)\lVert\mathbf e_{\mathcal H,j}^{k+1}\rVert_2
       +2C_{\mathcal D}h^{p+1}\lVert\mathbf e_j^k\rVert_2,\\
    \lVert\mathbf e_{\mathcal S,j+1}^{k+1}\rVert_2
    &\le(1+L_{\mathrm{tr}}h)\lVert\mathbf e_{\mathcal S,j}^{k+1}\rVert_2
       +(2C_{\mathcal D}h^{p+1}+L_\delta h+3C_rh^2)\lVert\mathbf e_j^k\rVert_2.
    \end{aligned}
  \]
  Both recurrences start from
  \(\mathbf e_{\mathcal H,k+1}^{k+1}=\mathbf e_{\mathcal S,k+1}^{k+1}=\mathbf0\).
  Forward substitution gives
  \[
    \lVert\mathbf e_{\mathcal H,m+1}^{k+1}\rVert_2
    \le2C_{\mathcal D}h^{p+1}
      \sum_{j=k+1}^{m}(1+L_{\mathrm{tr}}h)^{m-j}\lVert\mathbf e_j^k\rVert_2.
  \]
  Replacing the coefficient \(2C_{\mathcal D}h^{p+1}\) in this bound
  by \(2C_{\mathcal D}h^{p+1}+L_\delta h+3C_rh^2\) yields the bound for
  \(\lVert\mathbf e_{\mathcal S,m+1}^{k+1}\rVert_2\).
  In the displayed sum, the term \(j=k+1\) bounds the contribution from
  \(\mathbf e_{\mathcal H,k+2}^{k+1}\) after propagation through the
  remaining \(q_h-1\) intervals.
  Using
  \(\lVert\mathbf e_j^k\rVert_2\le
    \max_{k+1\le i\le m}\lVert\mathbf e_i^k\rVert_2\), we obtain
  \begin{equation}
    \label{eq:later-node-discrepancy-envelope}
    \lVert\mathbf e_{\mathcal H,m+1}^{k+1}\rVert_2
    \le2C_{\mathcal D}h^{p+1}
       \left(\max_{k+1\le j\le m}\lVert\mathbf e_j^k\rVert_2\right)
       \sum_{\ell=0}^{q_h-1}(1+L_{\mathrm{tr}}h)^\ell
  \end{equation}
  and
  \begin{equation}
    \label{eq:later-node-split-envelope}
    \begin{aligned}
    \lVert\mathbf e_{\mathcal S,m+1}^{k+1}\rVert_2
    &\le(2C_{\mathcal D}h^{p+1}+L_\delta h+3C_rh^2)\times\left(\max_{k+1\le j\le m}\lVert\mathbf e_j^k\rVert_2\right)
       \sum_{\ell=0}^{q_h-1}(1+L_{\mathrm{tr}}h)^\ell.
    \end{aligned}
  \end{equation}
  For \(q_h\le q_0\), each term of the sum is at most
  \((1+L_{\mathrm{tr}}h_0)^{q_0-1}\), so
  \[
    \sum_{\ell=0}^{q_h-1}(1+L_{\mathrm{tr}}h)^\ell
    \le q_0(1+L_{\mathrm{tr}}h_0)^{q_0-1}.
  \]
  Substituting this bound in \eqref{eq:later-node-discrepancy-envelope} and
  \eqref{eq:later-node-split-envelope} and using \(h^{p+1}\le h_0^ph\) and
  \(h^2\le h_0h\), we obtain
  \[
    \begin{aligned}
    \lVert\mathbf e_{\mathcal H,m+1}^{k+1}\rVert_2
    &\le2C_{\mathcal D}q_0(1+L_{\mathrm{tr}}h_0)^{q_0-1}\,h^{p+1}
       \max_{k+1\le j\le m}\lVert\mathbf e_j^k\rVert_2,\\
    \lVert\mathbf e_{\mathcal S,m+1}^{k+1}\rVert_2
    &\le(2C_{\mathcal D}h_0^p+L_\delta+3C_rh_0)\,
       q_0(1+L_{\mathrm{tr}}h_0)^{q_0-1}\,h
       \max_{k+1\le j\le m}\lVert\mathbf e_j^k\rVert_2,
    \end{aligned}
  \]
  which proves \eqref{eq:fixed-offset-form-envelopes}.
  For \(q_hh\le\tau_*\), the inequality
  \(1+L_{\mathrm{tr}}h\le e^{L_{\mathrm{tr}}h}\) and \(q_h\le\tau_*/h\) imply
  \[
    \sum_{\ell=0}^{q_h-1}(1+L_{\mathrm{tr}}h)^\ell
    \le q_he^{L_{\mathrm{tr}}q_hh}
    \le\frac{\tau_*e^{L_{\mathrm{tr}}\tau_*}}{h}.
  \]
  Substituting this bound in \eqref{eq:later-node-discrepancy-envelope} and
  \eqref{eq:later-node-split-envelope} and using \(h\le h_0\), we obtain
  \[
    \begin{aligned}
    \lVert\mathbf e_{\mathcal H,m+1}^{k+1}\rVert_2
    &\le2C_{\mathcal D}\tau_*e^{L_{\mathrm{tr}}\tau_*}\,h^p
       \max_{k+1\le j\le m}\lVert\mathbf e_j^k\rVert_2,\\
    \lVert\mathbf e_{\mathcal S,m+1}^{k+1}\rVert_2
    &\le(2C_{\mathcal D}h_0^p+L_\delta+3C_rh_0)\,
       \tau_*e^{L_{\mathrm{tr}}\tau_*}
       \max_{k+1\le j\le m}\lVert\mathbf e_j^k\rVert_2,
    \end{aligned}
  \]
  which proves \eqref{eq:physical-window-form-envelopes}.
  The constants in these bounds depend only on \(p\), \(h_0\), \(q_0\),
  \(\tau_*\), and the grid-independent constants of
  Assumption~\ref{ass:active-prefix-source-transport}.
\end{proof}

At \(m=k+1\), Theorem~\ref{prop:later-node-form-envelopes} gives the two
upper bounds \eqref{eq:fixed-offset-form-envelopes} at \(T_{k+2}\).
Theorem~\ref{thm:first-active-form-boundary} adds the lower bound for
\(\lVert\mathbf e_{\mathcal S,k+2}^{k+1}\rVert_2\)
and the error ratio in \eqref{eq:first-active-scales} under
Assumption~\ref{ass:first-active-source-separation}.
At subsequent time points, Theorem~\ref{prop:later-node-form-envelopes}
provides the separate upper bounds
\eqref{eq:fixed-offset-form-envelopes} and
\eqref{eq:physical-window-form-envelopes}.
Compared with the estimates \eqref{eq:fixed-offset-form-envelopes}
for a bounded number of coarse intervals, the window estimates
\eqref{eq:physical-window-form-envelopes} lose one power of \(h\)
because the window may contain \(O(h^{-1})\) contributing intervals.

\begin{example}
\label{ex:powers-of-h}
This example shows that the powers of \(h\) in
Theorems~\ref{thm:first-active-form-boundary}
and~\ref{prop:later-node-form-envelopes} cannot be improved in general.
Consider the system
\[
  u'=uv,\qquad v'=0,\qquad (u(0),v(0))=(a,b)
\]
with fixed \(a,b>0\) on the family of grids in
\eqref{eq:uniform-grid-family}.
Let \(P=\operatorname{diag}(1,0)\), which keeps \(u\) and removes \(v\).
Write \(\mathbf y=(u,v)^\top\).
The fine solver is the exact local solution map
\(\mathcal F_n(\mathbf y)=(ue^{hv},v)^\top\).
The coarse solver \(\mathcal G_n(\mathbf y)=(uR_p(hv),v)^\top\) is the
Taylor method of a fixed integer order \(p\ge1\), with
\(R_p(z)=\sum_{\ell=0}^{p}z^\ell/\ell!\).
Both solvers are the identity on the ranges of \(P\) and \(Q\).
Hence \(\mathcal H_{P,n}=\mathcal G_n\) and \(\mathcal S_{P,n}=I\), so
scheme~\eqref{eq:projection-parareal} coincides with classical
\emph{parareal} in this example.
The mixed difference \(\boldsymbol\delta_n(\mathbf y;P)=(uv,0)^\top\)
couples the growth of \(u\) to the component \(v\) removed by \(P\).
The fine solver in \(\mathcal S_{P,n}\) is evaluated at
\(P\mathbf y=(u,0)^\top\), where \(u\) does not grow.

Set \(r:=e^{bh}\) and \(g:=R_p(bh)\).
Take \(\mathbf U^0\) to be the sequential coarse initialization.
Since both solvers leave \(v\) unchanged, every approximation has the
second component \(b\).
Hence the errors are \(\mathbf e_n^0=(\varepsilon_n^0,0)^\top\),
\(\mathbf e_{\mathcal H,n}^1=(\varepsilon_{\mathcal H,n}^1,0)^\top\), and
\(\mathbf e_{\mathcal S,n}^1=(\varepsilon_{\mathcal S,n}^1,0)^\top\).
Since \(\mathbf u_n^\star=(ar^n,b)^\top\) and
\(\mathbf U_n^0=(ag^n,b)^\top\), we have \(\varepsilon_n^0=a(r^n-g^n)\).
One iteration step of each scheme yields
\[
  \varepsilon_{\mathcal H,n+1}^1
  =g\,\varepsilon_{\mathcal H,n}^1+(r-g)\,\varepsilon_n^0,
  \qquad
  \varepsilon_{\mathcal S,n+1}^1
  =\varepsilon_{\mathcal S,n}^1+(r-1)\,\varepsilon_n^0,
\]
with \(\varepsilon_{\mathcal H,0}^1=\varepsilon_{\mathcal S,0}^1=0\).
Since \(r>g>1\), these scalars are nonnegative and equal the
\(2\)-norms of the errors.
In scheme~\eqref{eq:projection-parareal}, the coefficient of the
old error is the difference \(r-g\) between the fine and coarse growth
factors.
In scheme~\eqref{eq:direct-split-iteration}, it is the growth increment
\(r-1\) over one coarse interval.
At \(T_2\), the first two ratios in \eqref{eq:first-active-scales} are
therefore \(r-g=b^{p+1}h^{p+1}/(p+1)!+O(h^{p+2})\) and
\(r-1=bh+O(h^2)\).
The recurrences also yield
\(\varepsilon_{\mathcal H,N}^1=a[r^N-g^N-N(r-g)g^{N-1}]\) and
\(\varepsilon_{\mathcal S,N}^1=a[r^N-1-(r-1)(g^N-1)/(g-1)]\).
As \(h\to0\) with \(Nh=T\), we obtain
\[
  \frac{\varepsilon_{\mathcal H,N}^1}
       {h^p\max_{1\le j\le N-1}\varepsilon_j^0}
  \to\frac{Tb^{p+1}}{2(p+1)!},
  \qquad
  \frac{\varepsilon_{\mathcal S,N}^1}
       {\max_{1\le j\le N-1}\varepsilon_j^0}
  \to1-\frac{1-e^{-bT}}{bT}>0.
\]
All states and segments involved lie in the rectangle
\([0,ae^{bT}]\times[0,b]\), together with their projections.
On a bounded neighborhood of this rectangle, Taylor expansion in \(h\)
yields the bounds of Assumption~\ref{ass:local-form-source-regularity} with
constants that depend only on \(a\), \(b\), \(T\), \(p\), and \(h_0\).
Assumption~\ref{ass:first-active-source-separation} holds with
\(c_\delta=b\) in \eqref{eq:first-active-source-separation}, since
\(\Delta\boldsymbol\delta=b\,\mathbf e_1^0\).
Assumption~\ref{ass:active-prefix-source-transport} holds with
\(L_{\mathrm{tr}}=be^{bh_0}\) in \eqref{eq:later-node-transport-stability}:
the differences there are \((g\varepsilon,0)^\top\) and
\((\varepsilon,0)^\top\) for an error \((\varepsilon,0)^\top\), and
\(g\le e^{bh}\le1+be^{bh_0}h\).
Hence the powers of \(h\) in \eqref{eq:first-active-scales},
\eqref{eq:fixed-offset-form-envelopes}, and
\eqref{eq:physical-window-form-envelopes} cannot be improved in general.
\end{example}

\section{Numerical experiments}
\label{sec:numerical-experiments}

In this section, we present numerical experiments for
projection-based \emph{parareal} with two reaction--diffusion problems.
In Section~\ref{subsec:numerics-projection-resolved}, we evaluate
the upper bounds of Theorem~\ref{lem:projection-dependent-one-sweep} and
Corollary~\ref{cor:one-step-special-cases} for
Algorithm~\ref{alg:projection-parareal} as the iteration proceeds.
We also compare the convergence behavior of the two schemes.
In Section~\ref{subsec:numerics-form-specific-corrected-errors}, we illustrate
the estimates of Theorems~\ref{thm:first-active-form-boundary}
and~\ref{prop:later-node-form-envelopes} for the errors after one iteration
step as \(h\to0\).
All numerical experiments were carried out in double precision
using MATLAB R2025b on a desktop computer (AMD Ryzen~7 9800X3D CPU, 4.70~GHz,
32~GB RAM, 64-bit Windows~11 Pro).

\subsection{Discretization and solvers}
\label{subsec:numerical-test-problems}

We impose periodic boundary conditions in space
and take the time interval to be \([0,T]\) with \(T=2\).
We approximate second spatial derivatives by second-order centered
differences on uniform spatial grids.
Sections~\ref{subsec:numerics-projection-resolved}
and~\ref{subsec:numerics-form-specific-corrected-errors} state the two problems
and their spatial grids.
The spatial grids do not change with \(h\), since the analysis concerns a
system of fixed dimension \(d\).

Both problems are autonomous: the right-hand side of
\eqref{eq:initial-value-problem} is
\(\mathbf f(\mathbf y,t)=A\mathbf y+\mathbf g(\mathbf y)\), where \(A\) is the
discretized diffusion matrix and the cubic polynomial \(\mathbf g\) collects
the reaction terms.
Hence \(\mathcal F_n\) and \(\mathcal G_n\) do not depend on \(n\).
The fine solver uses the classical fourth-order Runge--Kutta method
with \(32\) equal substeps of length \(h/32\) on each coarse
interval \([T_n,T_{n+1}]\).
The coarse solvers treat the diffusion term \(A\mathbf y\)
implicitly and the reaction term \(\mathbf g(\mathbf y)\) explicitly.
One implicit--explicit (IMEX) Euler step~\cite[Sec.~IV.4.1]{hundsdorfer2003numerical}
of length \(h\) gives
\[
  \widetilde{\mathbf y}:=(I-hA)^{-1}\bigl(\mathbf y+h\mathbf g(\mathbf y)\bigr).
\]
For \(p=1\), we take \(\mathcal G_n(\mathbf y)=\widetilde{\mathbf y}\).
For \(p=2\), one Crank--Nicolson--Heun step uses \(\widetilde{\mathbf y}\) as a
predictor and computes \(\mathcal G_n(\mathbf y)\) from
\[
  \frac{\mathcal G_n(\mathbf y)-\mathbf y}{h}
  =\frac12 A\bigl[\mathbf y+\mathcal G_n(\mathbf y)\bigr]
   +\frac12\bigl[\mathbf g(\mathbf y)+\mathbf g(\widetilde{\mathbf y})\bigr].
\]
The fine and coarse solvers are smooth on all of \(\mathbb R^d\),
so we take \(\Omega_n=\mathbb R^d\).
Every iteration step is then admissible, and every step of scheme
\eqref{eq:projection-parareal} satisfies
Assumption~\ref{ass:local-pathwise-regularity}.
The fine solver is a Runge--Kutta method of order four.
By the property recalled in
Section~\ref{subsec:projection-dependent-one-sweep},
\(\mathrm D\mathcal F_n(\mathbf y)\) therefore approximates the derivative of
the exact local solution map with respect to the initial value up to
\(O(h^5)\).
The local expansions of the two coarse solvers above show that
\(\mathrm D\mathcal G_n(\mathbf y)\) approximates it up to \(O(h^{p+1})\).
Hence
\[
  \bigl\lVert\mathrm D\mathcal D_n(\mathbf y)\bigr\rVert_2
  =\bigl\lVert\mathrm D\mathcal F_n(\mathbf y)
       -\mathrm D\mathcal G_n(\mathbf y)\bigr\rVert_2
  =O(h^{p+1}),\qquad p=1,2,
\]
uniformly on fixed bounded sets.
Thus \eqref{eq:derivative-source-bounds} in
Assumption~\ref{ass:form-derivative-bound} holds whenever \(\mathbf x\) and
\(\mathbf w\) lie in a fixed bounded set, with \(C_{\mathcal D}\) depending on
that set.
On these sets, the same expansions yield the bound
\(\lVert\mathrm D\mathbf r_n(\mathbf y)\rVert_2=O(h^2)\) in
Assumption~\ref{ass:form-expansion-bounds}.
The polynomial mixed difference in \eqref{eq:mixed-interaction} also satisfies
the bound \(\lVert\mathrm D\boldsymbol\delta_n(\mathbf y;P)\rVert_2=O(1)\) in
Assumption~\ref{ass:form-expansion-bounds} there.

\subsection{Convergence comparison as the iteration proceeds}
\label{subsec:numerics-projection-resolved}

The first experiment uses a two-dimensional Allen--Cahn equation
\begin{equation}
  \label{eq:numerical-allen-cahn}
  \partial_tu=0.001\,(\partial_{\xi\xi}u+\partial_{\eta\eta}u)+8(u-u^3),
  \qquad(\xi,\eta)\in[0,1)^2,\quad t\in(0,2],
\end{equation}
with \(u\) periodic in \(\xi\) and \(\eta\) and with the initial condition
\[
  \begin{aligned}
    u(\xi,\eta,0)&=0.55\cos(2\pi\xi+0.4)\cos(2\pi\eta+1.1)
       +0.30\sin\bigl(2\pi(2\xi+\eta)+0.7\bigr)\\
       &\quad+0.20\cos\bigl(2\pi(\xi-2\eta)+2.3\bigr).
  \end{aligned}
\]
The Allen--Cahn equation was introduced to model antiphase boundary
motion and domain coarsening~\cite{allen1979microscopic}.
Equations of this type are common test problems for exponential
integrators~\cite{kassam2005fourth}, maximum-principle-preserving
schemes~\cite{li2020cutoff}, and classical
\emph{parareal}~\cite{jin2025optimizingcoarse}.
The cubic reaction term makes the coarse solvers nonlinear, so the two
schemes can give different approximations when \(P^k\ne0\) and \(P^k\ne I\)
(Section~\ref{subsec:projection-based-parareal}).
They differ through the mixed difference \eqref{eq:mixed-interaction}, which
for this equation is proportional to the reaction coefficient~\(8\)
in~\eqref{eq:numerical-allen-cahn}.
We approximate \(\partial_{\xi\xi}u\) and \(\partial_{\eta\eta}u\) at the points
\((i/16,j/16)\) for \({i,j=0,\ldots,15}\), with the indices \(i\pm1\) and
\(j\pm1\) taken modulo \(16\).
The resulting system has dimension \(d=256\).
Its time-dependent vector solution \(\mathbf y(t)\in\mathbb R^{256}\) has the
entries \(y_{i+16j}(t)\approx u(i/16,j/16,t)\) for \(i,j=0,\ldots,15\).

We perform iterations with each of the two variants of
\emph{parareal} defined by schemes \eqref{eq:projection-parareal} and
\eqref{eq:direct-split-iteration} and report \(E_\infty^k\) as the iteration
proceeds.
Recall from Section~\ref{subsec:classical-parareal} that
\[
  E_\infty^k=\max_{0\le n\le N}\lVert\mathbf u_n^\star-\mathbf U_n^k\rVert_2,
\]
where \(\mathbf u^\star\) is the fine reference solution in
\eqref{eq:fine-reference}.
For these iterations, we fix \(N=32\), \(h=1/16\), and \(p=1\).
Both variants start from the same sequential coarse initialization.
We stop scheme \eqref{eq:projection-parareal} at the first \(k\) with
\(E_\infty^k\le10^{-12}\) and run scheme \eqref{eq:direct-split-iteration}
to the same \(k\), unless it overflows earlier.
Scheme \eqref{eq:direct-split-iteration} replaces \(\mathcal G_n\) in
\eqref{eq:classical-parareal} by \(\mathcal S_{P^k,n}\) as Krylov-enhanced
\emph{parareal} does, now with nonlinear solvers and an arbitrary orthogonal
projection (Section~\ref{subsec:krylov-enhanced-parareal}).
We use the following three projection choices.
\begin{itemize}
\item \emph{None.} We take \(P^k=0\), so both variants reduce to
classical \emph{parareal}.

\item \emph{Krylov-enhanced subspace.} We construct \(P^k\) for each
variant from the vectors spanning
\[
  \mathcal V^k=\operatorname{span}
  \{\mathbf U_j^\ell:0\le\ell\le k,\ 0\le j\le N\},
\]
following the construction of Krylov-enhanced \emph{parareal}~\cite{gander2008krylov}
recalled in Section~\ref{subsec:krylov-enhanced-parareal}.
In the computation, we form the matrix with columns
\(\mathbf U_j^\ell\).
Then \(P^k\) is the orthogonal projection onto the span of its left singular
vectors whose singular values exceed \(10^{-12}\) times the largest one.

\item \emph{Proper orthogonal decomposition (POD)}~\cite{sirovich1987coherent}.
We construct \(P^k\) for all \(k\) and both variants from eight
sequences generated by recurrence~\eqref{eq:fine-reference}.
These sequences start from the initial conditions
\[
  \begin{aligned}
    u(\xi,\eta,0)&=0.55(1+0.2\sigma_1)\cos(2\pi\xi+0.4)\cos(2\pi\eta+1.1)\\
       &\quad+0.30(1+0.2\sigma_2)\sin\bigl(2\pi(2\xi+\eta)+0.7\bigr)\\
       &\quad+0.20(1+0.2\sigma_3)\cos\bigl(2\pi(\xi-2\eta)+2.3\bigr),\\
       &\qquad(\sigma_1,\sigma_2,\sigma_3)\in\{-1,1\}^3.
  \end{aligned}
\]
The initial condition of the problem corresponds to
\(\sigma_1=\sigma_2=\sigma_3=0\), so none of the eight sequences is the fine
reference solution \(\mathbf u^\star\).
We form the \(256\times264\) matrix with the vectors of these sequences as
columns.
Then \(P^k\) is the orthogonal projection onto the span of its \(16\) leading
left singular vectors.
\end{itemize}

We now evaluate the upper bounds of
Theorem~\ref{lem:projection-dependent-one-sweep} and
Corollary~\ref{cor:one-step-special-cases} for
Algorithm~\ref{alg:projection-parareal}.
For each projection choice and each iteration step \(k\to k+1\), write
\(P=P^k\) and \(Q=I-P\), and set
\begin{equation}
  \label{eq:numerical-projection-maxima}
  m_Q^k:=\max_{n\in\mathcal A_k}\lVert Q\mathbf e_n^k\rVert_2,
  \qquad
  m_P^k:=\max_{n\in\mathcal A_k}\lVert P\mathbf e_n^k\rVert_2.
\end{equation}
Recall from Section~\ref{subsec:projection-dependent-one-sweep}
that
\[
  U_k^\perp(P)=\max_{n\in\mathcal A_k}\lVert Q\mathbf U_n^k\rVert_2.
\]
For brevity, we write \(U_k^\perp\) for \(U_k^\perp(P)\) in this
subsection.
For \(E_\infty^k>0\), the definitions imply
\(m_Q^k=E_\infty^k\theta_Q^k(P)\) and
\(m_P^k=E_\infty^k\theta_P^k(P)\).

We compute numerical estimates \(\widehat\mu_k\) and \(\widehat\nu_k\)
of the smallest constants \(\mu_k\) and \(\nu_k\) in \eqref{eq:source-path-bounds},
and \(\widehat\gamma_k\) of the propagation factor \(\gamma_k\) in
\eqref{eq:causal-gain}.
The Jacobians \(\mathrm D\mathcal F_n(\mathbf y)\) and
\(\mathrm D\mathcal G_n(\mathbf y)\) are computed by differentiating the
numerical solver formulas with respect to \(\mathbf y\).
\begin{samepage}%
The set \(\mathcal T_k\) consists of \(0\), \(1/2\), \(1\), and \(50\)
values drawn independently from the uniform distribution on \([0,1]\) with a
fixed seed.
For \(\widehat\mu_k\) and \(\widehat\nu_k\), we sample
\(\mathbf z_n^k(\tau)=\mathbf U_n^k+\tau\mathbf e_n^k\) at \(\tau\in\mathcal T_k\)
and set
\[
  \begin{aligned}
    \widehat\mu_k
    &:=\max_{\substack{n\in\mathcal A_k\\\tau\in\mathcal T_k}}
      \lVert\mathrm D\mathcal D_n(\mathbf z_n^k(\tau))\rVert_2,\\[2pt]
    \widehat\nu_k
    &:=\max_{\substack{n\in\mathcal A_k\\\tau\in\mathcal T_k}}
      \frac{\lVert\mathrm D\mathcal D_n(\mathbf z_n^k(\tau))
            -\mathrm D\mathcal D_n(P\mathbf z_n^k(\tau))\rVert_2}
           {\lVert Q\mathbf z_n^k(\tau)\rVert_2}.
  \end{aligned}
\]
\end{samepage}%
Samples with
\[
  \lVert Q\mathbf z_n^k(\tau)\rVert_2
  \le10^{3}\cdot2^{-53}\max\{1,\lVert\mathbf z_n^k(\tau)\rVert_2\}
\]
are omitted from the maximum in \(\widehat\nu_k\), to avoid quotients dominated
by rounding errors.
For \(\widehat\gamma_k\), the averaged Jacobians \(B_n^{k+1}\) in
\eqref{eq:realized-secant} are approximated by \(16\)-point
Gauss\nobreakdash--Legendre quadrature along
\(\boldsymbol\zeta_n^{k+1}(\tau)=\mathbf U_n^{k+1}+\tau\mathbf e_n^{k+1}\).
Substituting these approximations into \eqref{eq:ordered-products} and
\eqref{eq:causal-gain} gives \(\widehat\gamma_k\).

For the three projection choices, we use the upper bounds of the
following results.
\begin{itemize}
\item \emph{None.} Corollary~\ref{cor:zero-projection},
since \(P^k=0\).
\item \emph{Krylov-enhanced subspace.}
Corollary~\ref{cor:old-state-projection}.
\item \emph{POD.} Theorem~\ref{lem:projection-dependent-one-sweep},
since \(P^k\mathbf U_n^k\ne\mathbf U_n^k\) for some \(n\in\mathcal A_k\).
\end{itemize}
For \(P^k=0\) and for the Krylov-enhanced subspace, we use
\(\widehat C_{\mathcal D}h^{p+1}\) in place of \(\mu_k\) with
\[
  \widehat C_{\mathcal D}:=h^{-(p+1)}\max_{\ell}\widehat\mu_\ell.
\]
Here \(\ell\) ranges over the iteration steps performed.
We compute \(\widehat C_{\mathcal D}\) separately for each of these two
projection choices.
We substitute these estimates into the three upper bounds and define
\begingroup
\setlength{\abovedisplayskip}{6pt}
\setlength{\belowdisplayskip}{6pt}
\begin{equation}
  \label{eq:numerical-one-step-estimate}
  \widehat b_{k+1}:=\widehat\gamma_k
  \begin{cases}
    \widehat C_{\mathcal D}h^{p+1}m_Q^k,
      &\text{None},\\[2pt]
    m_Q^k\bigl(\widehat C_{\mathcal D}h^{p+1}
                      +\tfrac12\widehat\nu_k E_\infty^k\bigr),
      &\text{Krylov-enhanced subspace},\\[2pt]
    \widehat\mu_k m_Q^k+\widehat\nu_k U_k^\perp m_P^k
       +\tfrac12\widehat\nu_k m_Q^k m_P^k,
      &\text{POD}.
  \end{cases}
\end{equation}
\endgroup%
Since \(\widehat\mu_k\) and \(\widehat\nu_k\) are maxima over finitely
many values of \(\tau\) and the averaged Jacobians are approximated by
quadrature, \(\widehat b_{k+1}\) is an estimated upper bound on
\(E_\infty^{k+1}\) rather than a rigorous one.
For the Krylov-enhanced subspace, the truncation makes
\(P^k\mathbf U_n^k=\mathbf U_n^k\) hold only approximately, but the term
\(\widehat\gamma_k\widehat\nu_kU_k^\perp m_P^k\) that
Corollary~\ref{cor:old-state-projection} drops from
Theorem~\ref{lem:projection-dependent-one-sweep} stays below
\(10^{-9}\,\widehat b_{k+1}\) at every iteration step.
\begin{figure}[!tb]
  \centering
  \includegraphics[width=\textwidth]{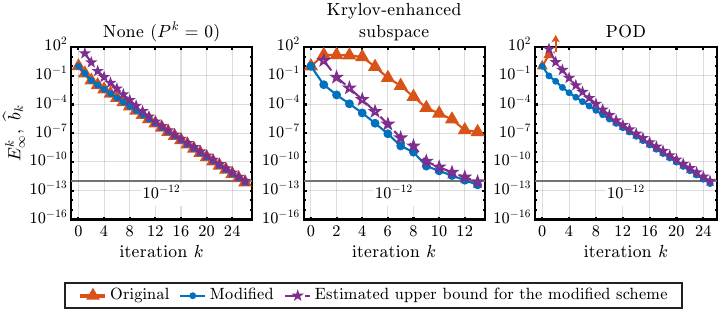}
  \caption{Values of \(E_\infty^k\) for the two schemes and of the estimated
upper bound \(\widehat b_k\) in \eqref{eq:numerical-one-step-estimate} for the
modified scheme in the Allen--Cahn experiment.}
  \label{fig:pillar-one-combined-estimates}
\end{figure}

Figure~\ref{fig:pillar-one-combined-estimates} shows \(E_\infty^k\)
of the two schemes and the estimated upper bound \(\widehat b_k\) for
scheme~\eqref{eq:projection-parareal}, with one column for each projection
choice.
In the legends of
Figures~\ref{fig:pillar-one-combined-estimates}\nobreakdash--\ref{fig:pillar-two-window-envelopes},
``Original'' denotes scheme~\eqref{eq:direct-split-iteration} and
``Modified'' denotes scheme~\eqref{eq:projection-parareal}.
The horizontal line is the stopping tolerance \(10^{-12}\) for scheme
\eqref{eq:projection-parareal}.

We summarize the comparisons in
Figure~\ref{fig:pillar-one-combined-estimates} as follows.
\begin{itemize}
\item The computed \(E_\infty^k\) of
scheme~\eqref{eq:projection-parareal} is consistent with the upper bounds of
Corollary~\ref{cor:zero-projection} for \(P^k=0\), of
Corollary~\ref{cor:old-state-projection} for the Krylov-enhanced subspace,
and of Theorem~\ref{lem:projection-dependent-one-sweep} for POD.
At every \(k\ge1\) shown, the estimated upper bound \(\widehat b_k\) exceeds
\(E_\infty^k\).
In each column, the ratio \(\widehat b_k/E_\infty^k\) decreases as
\(k\) increases.
It lies between \(10^2\) and \(10^3\) at \(k=1\) and below \(2.1\) at the
last \(k\) shown.
\item The two schemes give different \(E_\infty^k\) with the two
nonzero projections.
For \(P^k=0\), both schemes reduce to classical \emph{parareal} and give the
same \(E_\infty^k\).
The comparisons with nonzero projections are observations, since
Theorem~\ref{lem:projection-dependent-one-sweep} covers only
scheme~\eqref{eq:projection-parareal} and the analysis of Krylov-enhanced
\emph{parareal} in~\cite{gander2008krylov} concerns linear problems.
With the Krylov-enhanced subspace, scheme~\eqref{eq:projection-parareal}
gives smaller \(E_\infty^k\) than scheme~\eqref{eq:direct-split-iteration}
at every \(k\ge1\) shown.
For scheme~\eqref{eq:direct-split-iteration}, \(E_\infty^k\) exceeds
\(E_\infty^0\) for \(k=1,\ldots,4\) and is still above \(10^{-7}\) when
scheme~\eqref{eq:projection-parareal} stops at \(k=13\).
With POD, \(E_\infty^1\) of scheme~\eqref{eq:direct-split-iteration} also
exceeds \(E_\infty^0\).
In the iteration step \(1\to2\), floating-point overflow occurs when the
explicit fine solver computes \(\mathcal F_{18}(P^1\mathbf U_{18}^2)\).
The computed \(\mathbf U_n^2\) are therefore not finite for \(n\ge19\).
A dotted segment and an upward arrow at \(k=2\) mark this overflow.
\item The projection strongly affects the convergence of
scheme~\eqref{eq:projection-parareal}.
With POD, \(E_\infty^k\) of scheme~\eqref{eq:projection-parareal} is close to
that for \(P^k=0\) at every \(k\) shown.
With the Krylov-enhanced subspace, \(E_\infty^k\) of
scheme~\eqref{eq:projection-parareal} decreases much faster.
\end{itemize}

\FloatBarrier

\subsection{Asymptotic error comparison}
\label{subsec:numerics-form-specific-corrected-errors}

The second experiment uses a spatial FitzHugh--Nagumo system
\[
  \begin{aligned}
    \partial_tu&=0.001\,\partial_{\xi\xi}u
        +10\bigl[u(1-u)(u-0.1)-v\bigr],\\
    \partial_tv&=0.0001\,\partial_{\xi\xi}v+u-2v,
  \end{aligned}
  \qquad\xi\in[0,1),\quad t\in(0,2],
\]
with \(u\) and \(v\) periodic in \(\xi\).
Writing
\(\operatorname{dist}(\xi,c):=\min_{\ell\in\mathbb Z}|\xi-c-\ell|\)
for the distance on the periodic spatial interval, we take
the initial condition
\[
  \begin{aligned}
    u(\xi,0)&=0.2+0.8\exp\bigl(-80\operatorname{dist}(\xi,0.3)^2\bigr),\\
    v(\xi,0)&=0.02\sin(2\pi\xi).
  \end{aligned}
\]
The FitzHugh--Nagumo system originates in models of nerve
membrane~\cite{fitzhugh1961impulses} and of pulse propagation along a nerve
axon~\cite{nagumo1962active}.
Spatial systems of this type are common test problems for model-order
reduction~\cite{benner2021polynomial}, exponential
integrators~\cite{caliari2024directional}, and classical
\emph{parareal}~\cite{gattiglio2025nngparareal}.
We approximate \(\partial_{\xi\xi}u\) and \(\partial_{\xi\xi}v\)
at the points \(i/32\) for \({i=0,\ldots,31}\), with the indices
\(i\pm1\) taken modulo \(32\).
The resulting system has dimension \(d=2\cdot32=64\).
Its time-dependent vector solution \(\mathbf y(t)\in\mathbb R^{64}\)
has the entries \(y_i(t)\approx u(i/32,t)\) and \(y_{32+i}(t)\approx v(i/32,t)\)
for \(i=0,\ldots,31\).

We illustrate the asymptotic error estimates of
Theorems~\ref{thm:first-active-form-boundary} and
\ref{prop:later-node-form-envelopes} using the errors after one iteration step.
We use \(N=2^4,2^5,\ldots,2^{10}\), \(h=2/N\), and \(p=1,2\).
Let \(\boldsymbol\phi,\boldsymbol\psi\in\mathbb R^{32}\) have the entries
\(\phi_i=\sin(2\pi i/32)\) and \(\psi_i=\cos(2\pi i/32)\) for \(i=0,\ldots,31\).
Since \(\boldsymbol\phi^\top\boldsymbol\psi=0\) and
\(\lVert\boldsymbol\phi\rVert_2^2=\lVert\boldsymbol\psi\rVert_2^2=16\), the matrix
\[
  \Pi:=\frac{1}{16}\bigl(\boldsymbol\phi\boldsymbol\phi^\top
      +\boldsymbol\psi\boldsymbol\psi^\top\bigr)
\]
is the orthogonal projection onto \(\operatorname{span}\{\boldsymbol\phi,\boldsymbol\psi\}\).
For all \(h\), we take \(P=\operatorname{diag}(\Pi,\Pi)\).
This projection applies \(\Pi\) to the grid values of \(u\) and to those of
\(v\).
Projections onto trigonometric modes have also been used in
\emph{parareal}~\cite{dai2013projection,reynoldsbarredo2013analytic}.
From the sequential coarse initialization \(\mathbf U^0\), one iteration step of
scheme \eqref{eq:projection-parareal} and one of scheme
\eqref{eq:direct-split-iteration} with the same \(P\) give
\(\mathbf U_{\mathcal H}^1\) and \(\mathbf U_{\mathcal S}^1\).
This is the case \(k=0\) of Theorems~\ref{thm:first-active-form-boundary}
and~\ref{prop:later-node-form-envelopes}.
We examine the errors at \(T_2=2h\), \(T_3=3h\), and \(T_N=Nh=2\).
Only \(T_N\) is fixed as \(h\to0\).

Since the linear terms of \(\mathbf f\) cancel in
\eqref{eq:mixed-interaction}, the mixed difference of this system has the
entries
\[
  [\boldsymbol\delta_n(\mathbf y;P)]_i=10\,[P\mathbf y]_i\,[Q\mathbf y]_i\,(2.2-3y_i),
  \qquad
  [\boldsymbol\delta_n(\mathbf y;P)]_{32+i}=0,
  \qquad i=0,\ldots,31.
\]%
For all tested grids and both values of \(p\), the old error
\(\mathbf e_1^0\) is nonzero and the quotient
\[
  \frac{\lVert\boldsymbol\delta_1(\mathbf u_1^\star;P)
        -\boldsymbol\delta_1(\mathbf U_1^0;P)\rVert_2}
       {\lVert\mathbf e_1^0\rVert_2}
\]
lies between \(6.8\) and \(9.1\).
On the three finest grids, the quotient lies between \(7.27\)
and \(7.29\) for \(p=1\) and between \(9.03\) and \(9.05\) for \(p=2\).
These values support the lower bound
\eqref{eq:first-active-source-separation} in
Assumption~\ref{ass:first-active-source-separation} on the tested grids.
The uniform upper bounds in
Assumptions~\ref{ass:first-active-source-separation}
and~\ref{ass:active-prefix-source-transport} hold if
\(\mathbf u_j^\star\), \(\mathbf U_j^0\), \(\mathbf U_{\mathcal H,j}^1\), and
\(\mathbf U_{\mathcal S,j}^1\) remain in a common bounded set for
\(1\le j\le N-1\) as \(h\to0\).
Section~\ref{subsec:numerical-test-problems} then gives the bounds of
Assumption~\ref{ass:form-expansion-bounds} uniformly at the pairs
\((\mathbf u_j^\star,\mathbf U_j^0)\).
The solver formulas also imply \(\mathrm D\mathcal H_{P,n}(\mathbf y)=I+O(h)\)
and \(\mathrm D\mathcal S_{P,n}(\mathbf y)=I+O(h)\) uniformly on fixed bounded
sets.
By the mean value inequality along the segments joining
\(\mathbf U_{\mathcal H,j}^1\) and \(\mathbf U_{\mathcal S,j}^1\) to
\(\mathbf u_j^\star\), these derivative bounds imply the propagation bounds
\eqref{eq:later-node-transport-stability} in
Assumption~\ref{ass:active-prefix-source-transport}.
The maximum of the \(2\)-norms of these four vectors over \(1\le j\le N-1\)
and \(p=1,2\) is at most \(2.76\) on every tested grid and decreases as \(h\)
decreases.
Together with the stated setup, these observations support the assumptions of
Theorems~\ref{thm:first-active-form-boundary}
and~\ref{prop:later-node-form-envelopes} on the tested grids.

\begin{figure}[!b]
  \centering
  \includegraphics[width=\textwidth]{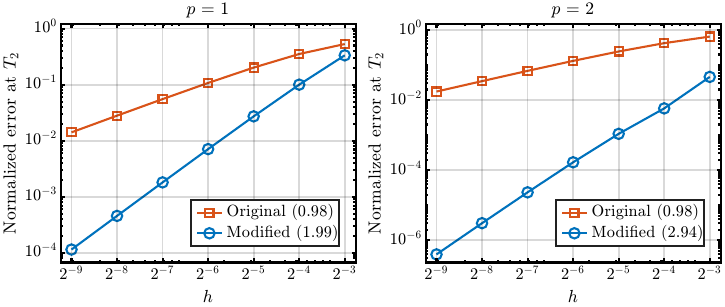}
  \caption{Normalized errors \eqref{eq:numerical-first-point-errors}
at the first potentially inexact time point \(T_2\) in the
FitzHugh--Nagumo experiment for coarse solvers of different orders
\(p\), with fitted slopes given in parentheses in the legends.
These slopes are close to \(1\) for ``Original'' and to \(p+1\) for
``Modified''.}
  \label{fig:pillar-two-first-active-form-boundary}
\end{figure}
At \(T_2\), we normalize the error norms of both schemes by the old
error norm \(\lVert\mathbf e_1^0\rVert_2\).
Theorem~\ref{thm:first-active-form-boundary} gives an
\(O(h^{p+1})\) upper bound for the normalized error of
scheme~\eqref{eq:projection-parareal} and a two-sided \(\Theta(h)\) estimate
for that of scheme~\eqref{eq:direct-split-iteration}.
Since \(q_h=2\) at \(T_3\) and \(q_hh=T-h\) at \(T_N\),
Theorem~\ref{prop:later-node-form-envelopes} gives the upper bounds
\eqref{eq:fixed-offset-form-envelopes} at \(T_3\) and
\eqref{eq:physical-window-form-envelopes} at \(T_N\).
For scheme~\eqref{eq:projection-parareal}, the powers of \(h\) in
these bounds are \(h^{p+1}\) at \(T_3\) and \(h^p\) at \(T_N\).
For scheme~\eqref{eq:direct-split-iteration}, they are \(h\) at \(T_3\) and
\(h^0\) at \(T_N\).

At the first potentially inexact time point \(T_2\), we plot in
Figure~\ref{fig:pillar-two-first-active-form-boundary} the normalized errors
\begin{equation}
  \label{eq:numerical-first-point-errors}
  \frac{\lVert\mathbf e_{\mathcal H,2}^1\rVert_2}
       {\lVert\mathbf e_1^0\rVert_2},
  \qquad
  \frac{\lVert\mathbf e_{\mathcal S,2}^1\rVert_2}
       {\lVert\mathbf e_1^0\rVert_2}.
\end{equation}
With \(m=2\) for \(T_3\) and \(m=N-1\) for \(T_N\), we plot in
Figures~\ref{fig:pillar-two-later-node-envelopes}
and~\ref{fig:pillar-two-window-envelopes} the error norms at
\(T_{m+1}\) normalized by the largest old error norm at \(T_1,\ldots,T_m\)
\begin{equation}
  \label{eq:numerical-later-point-errors}
  \frac{\lVert\mathbf e_{\mathcal H,m+1}^1\rVert_2}
       {\max_{1\le j\le m}\lVert\mathbf e_j^0\rVert_2},
  \qquad
  \frac{\lVert\mathbf e_{\mathcal S,m+1}^1\rVert_2}
       {\max_{1\le j\le m}\lVert\mathbf e_j^0\rVert_2}
  \qquad(m=2,N-1).
\end{equation}
The slope of each curve is obtained from a least-squares fit in
log--log coordinates to its values at \(h=2^{-9},2^{-8},2^{-7}\).
The legends of
Figures~\ref{fig:pillar-two-first-active-form-boundary}\nobreakdash--\ref{fig:pillar-two-window-envelopes}
give these fitted slopes in parentheses.

\begin{figure}[!t]
  \centering
  \includegraphics[width=\textwidth]{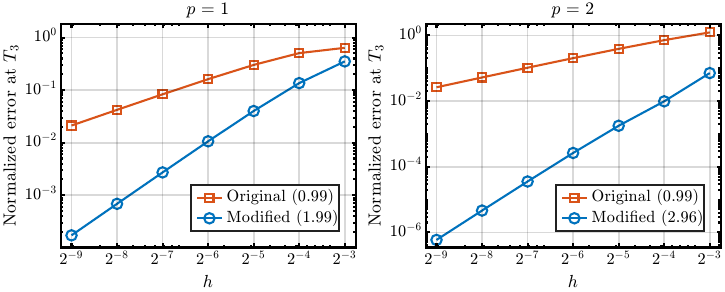}
  \caption{Normalized errors \eqref{eq:numerical-later-point-errors}
with \(m=2\) at \(T_3\) in the FitzHugh--Nagumo experiment
for coarse solvers of different orders \(p\), with fitted slopes
given in parentheses in the legends.
These slopes are close to \(1\) for ``Original'' and to \(p+1\) for
``Modified''.}
  \label{fig:pillar-two-later-node-envelopes}
\end{figure}
\begin{figure}[!t]
  \centering
  \includegraphics[width=\textwidth]{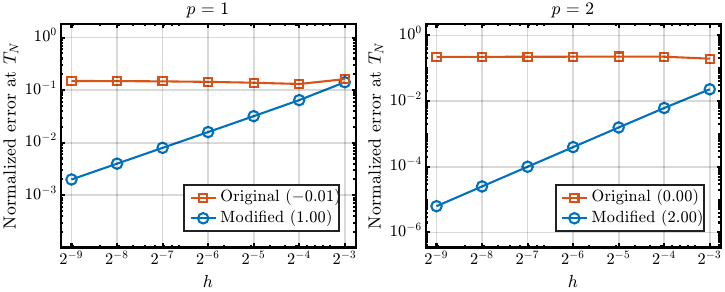}
  \caption{Normalized errors \eqref{eq:numerical-later-point-errors}
with \(m=N-1\) at \(T_N\) in the FitzHugh--Nagumo experiment for coarse solvers
of different orders \(p\), with fitted slopes given in parentheses in the
legends. These slopes are close to \(0\) for ``Original'' and to \(p\) for
``Modified''.}
  \label{fig:pillar-two-window-envelopes}
\end{figure}

We first compare the normalized errors at \(T_2\) with the estimates
of Theorem~\ref{thm:first-active-form-boundary}.
For scheme~\eqref{eq:projection-parareal}, the normalized error decreases like
\(h^{p+1}\).
This is consistent with the \(O(h^{p+1})\) upper bound.
For scheme~\eqref{eq:direct-split-iteration}, the fitted slopes are close to
\(1\) for both values of \(p\).
This supports the two-sided \(\Theta(h)\) estimate.

We next compare the normalized errors at \(T_3\) and \(T_N\) with the
upper bounds of Theorem~\ref{prop:later-node-form-envelopes}.
The bounds \eqref{eq:fixed-offset-form-envelopes} apply at \(T_3\), where the
number \(q_h=2\) of coarse intervals is fixed.
The normalized error decreases like \(h^{p+1}\) for
scheme~\eqref{eq:projection-parareal} and like \(h\) for
scheme~\eqref{eq:direct-split-iteration}.
This is consistent with the powers \(h^{p+1}\) and \(h\) in
\eqref{eq:fixed-offset-form-envelopes}.
The bounds \eqref{eq:physical-window-form-envelopes} apply at \(T_N\), where
the \(q_h=N-1\) coarse intervals lie in the fixed time window \([0,T]\).
The normalized error decreases like \(h^p\) for
scheme~\eqref{eq:projection-parareal} and stays nearly constant for
scheme~\eqref{eq:direct-split-iteration}.
This is consistent with the powers \(h^p\) and \(h^0\) in
\eqref{eq:physical-window-form-envelopes}.
On the tested grids, the error norm \(\lVert\mathbf e_{\mathcal S,N}^{1}\rVert_2\)
and the denominator in \eqref{eq:numerical-later-point-errors} for \(m=N-1\)
both decrease like \(h^p\).

\section{Conclusion}
\label{sec:conclusion}

We studied a projection-based \emph{parareal} algorithm for nonlinear
initial-value problems with a modified iteration scheme.
On a fixed time grid, we proved a convergence estimate for each iteration step of the algorithm.
The estimate shows how the projection affects the convergence of the iteration
through the components of the old error outside its range and through the
distance from the old approximations to this range.
As \({h\to0}\), we proved that the error of the
modified scheme at the first potentially inexact time point is of higher order
than that of the scheme of Krylov-enhanced \emph{parareal} under mild assumptions.
At later time points, we derived separate upper bounds for the two schemes
over a bounded number of coarse intervals and over a fixed time window.
Numerical experiments with the Allen--Cahn equation and a FitzHugh--Nagumo
system are consistent with our theoretical analysis.
Future work may replace the fine solver on the range of the
projection in the sequential update by a reduced model, as in the
reduced-order coarse solvers of~\cite{he2010reduced,chen2014parareal}.

\section*{Acknowledgments}
The authors used large language models (LLMs) in the research
process and in the preparation of this manuscript.
\begin{itemize}
\item In the research process, LLMs were used to assist in designing
code for the numerical experiments and to check the theoretical results.
\item In the preparation of this manuscript, LLMs were used to
improve the readability of the text and to check the references.
\end{itemize}
The authors assume responsibility for all content.

\bibliographystyle{siamplain}
\bibliography{references}
\end{document}